\documentclass[11pt,a4paper]{amsart}

\usepackage{amsmath,amssymb,amsthm,amsfonts,enumerate,mathtools,pinlabel,float,stmaryrd}
\usepackage[mathscr]{eucal}
\usepackage{amsbsy}
\usepackage{graphicx}
\usepackage{caption}
\usepackage{subcaption}
\usepackage{cite}
\usepackage{array}
\usepackage{xcolor}
\usepackage{calc}
\usepackage{comment}
\usepackage{soul}
\usepackage[myheadings]{fullpage}
\usepackage{hyperref}
\hypersetup{
    colorlinks=true,   
    linkcolor=blue,
    citecolor=blue,
}

\newtheorem{cor*}{Corollary}
\newtheorem{prop*}{Proposition}
\newtheorem{thm*}{Theorem}

\newtheorem{theorem}{Theorem}[section]
\newtheorem{cor}{Corollary}[theorem]
\newtheorem{prop}[theorem]{Proposition}
\newtheorem{lem}[theorem]{Lemma}

\theoremstyle{definition}
\newtheorem{defn}[theorem]{Definition}
\newtheorem{exmp}[theorem]{Example}
\newtheorem{rem}[theorem]{Remark}

\newcommand{\f}{\mathfrak{F}}
\newcommand{\F}{\mathcal{F}}
\newcommand{\G}{\mathcal{G}}
\newcommand{\Z}{\mathbb{Z}}
\newcommand{\N}{\mathbb{N}}
\newcommand{\R}{\mathbb{R}}
\newcommand{\Orb}{\mathcal{O}}

\newcommand{\C}{\mathcal{C}}

\newcommand{\h}{\mathcal{H}}
\newcommand{\Le}{\mathcal{L}}

\DeclareMathOperator{\aut}{Aut}
\DeclareMathOperator{\lcm}{lcm}
\DeclareMathOperator{\map}{Mod}

\DeclareMathOperator{\homeo}{Homeo^{+}}

\everymath{\displaystyle}

\makeatletter
\@namedef{subjclassname@2020}{\textup{2020} Mathematics Subject Classification}
\makeatother

\begin{document}

\title[Infinite metacyclic subgroups of the mapping class group]{Infinite metacyclic subgroups of \\the mapping class group}

\author{Pankaj Kapari}
\address{Department of Mathematics\\
Indian Institute of Science Education and Research Bhopal\\
Bhopal Bypass Road, Bhauri \\
Bhopal 462 066, Madhya Pradesh\\
India}
\email{pankajkapri02@gmail.com}

\author{Kashyap Rajeevsarathy}
\address{Department of Mathematics\\
Indian Institute of Science Education and Research Bhopal\\
Bhopal Bypass Road, Bhauri \\
Bhopal 462 066, Madhya Pradesh\\
India}
\email{kashyap@iiserb.ac.in}
\urladdr{https://home.iiserb.ac.in/$_{\widetilde{\phantom{n}}}$kashyap/}

\author{Apeksha Sanghi}
\address{Department of Mathematics\\
Indian Institute of Science Education and Research Bhopal\\
Bhopal Bypass Road, Bhauri \\
Bhopal 462 066, Madhya Pradesh\\
India}
\email{apekshasanghi93@gmail.com}

\subjclass[2020]{Primary 57K20, Secondary 57M60}

\keywords{surface, pseudo-periodic mapping class, pseudo-Anosov mapping class, metacyclic group}

\begin{abstract}
For $g\geq 2$, let $\text{Mod}(S_g)$ be the mapping class group of the closed orientable surface $S_g$ of genus $g$. In this paper, we provide necessary and sufficient conditions for a pair of elements in $\text{Mod}(S_g)$ to generate an infinite metacyclic subgroup. In particular, we provide necessary and sufficient conditions under which a pseudo-Anosov mapping class generates an infinite metacyclic subgroup of $\text{Mod}(S_g)$ with a nontrivial periodic mapping class. As applications of our main results, we establish the existence of infinite metacyclic subgroups of $\text{Mod}(S_g)$ isomorphic to $\mathbb{Z}\rtimes \mathbb{Z}_m, \mathbb{Z}_n \rtimes \mathbb{Z}$, and $\mathbb{Z} \rtimes \mathbb{Z}$. Furthermore, we derive bounds on the order of a nontrivial periodic generator of an infinite metacyclic subgroup of $\text{Mod}(S_g)$ that are realized. Finally, we show that the centralizer of an irreducible periodic mapping class $F$ is either $\langle F\rangle$ or $\langle F\rangle \times \langle i\rangle$, where $i$ is a hyperelliptic involution. 
\end{abstract}

\maketitle

\section{Introduction}
\label{sec:intro}
Let $\map(S_g)$ be the mapping class group of the closed orientable surface $S_g$ of genus $g\geq 2$. A metacyclic group is an extension of a cyclic group by a cyclic group. Given $F,G\in \map(S_g)$, it is natural to ask the following question: Can one derive necessary and sufficient conditions under which $F$ and $G$ generate a metacyclic subgroup of $\map(S_g)$? Ivanov (see \cite[Theorem 7.5A]{ivanov2}) derived necessary and sufficient conditions under which two pure mapping classes commute in $\map(S_g)$. Subsequently, the finite abelian subgroups of $\map(S_g)$ have been extensively studied~\cite{wootton,harvey,ab_bound}. Furthermore, in \cite{dhanwani, sanghi1, sanghi2}, the question (posed earlier) has been answered in the affirmative for finite metacyclic subgroups of $\map(S_g)$ up to conjugacy of their generators. Moreover, it was shown in \cite{sanghi1} that for $g\geq 5$, $\map(S_g)$ has an infinite metacyclic subgroup generated by a bounding pair map and an involution. Taking inspiration from these works, in this paper, we settle this question for infinite metacyclic subgroups of $\map(S_g)$. 

A \textit{multicurve} in $S_g$ is a nonempty collection of isotopy classes of pairwise disjoint essential simple closed curves. A left-handed (or positive) Dehn twist about a simple closed curve $c$ will be denoted by $T_c$. Given a multicurve $C=\{c_1,c_2, \dots, c_{\ell}\}$ in $S_g$ and nonzero integers $q_i$, for $ 1 \leq i \leq \ell$, a mapping class of the form $T_{c_1}^{q_1}T_{c_2}^{q_2}\cdots T_{c_{\ell}}^{q_{\ell}}$ is said to be a \textit{multitwist} about $C$. The Nielsen-Thurston classification \cite{thurston2} asserts that each mapping class in $\map(S_g)$ is either periodic, reducible, or pseudo-Anosov. Furthermore, a pseudo-Anosov mapping class is neither periodic nor reducible. The intersection of all maximal reduction systems of a reducible mapping class $F$ is called its \textit{canonical reduction system}, which we denote by $\C(F)$.

Let $F\in \map(S_g)$ be an infinite order reducible mapping class. Let $\C(F)=\{c_1,c_2,\dots,c_{\ell}\}$ be the canonical reduction system for $F$ and $N$ be an $F$-invariant closed regular neighborhood of $\C(F)$. Let $n$ be the least positive integer such that $F^n$ fixes each path component of $\overline{S_g\setminus N}\cup N$. Then, as a consequence of the Nielsen-Thurston classification \cite{thurston2}, there exist $s\in \N \cup \{0\}$ and $q_i\in \Z$ such that
\begin{equation}
\label{eqn:can_form}
F^n=T_{c_1}^{q_1}T_{c_2}^{q_2}\cdots T_{c_{\ell}}^{q_{\ell}}\eta_1(F_1)\eta_2(F_2)\cdots \eta_s(F_s)
\end{equation}
with $F_i\in \map(R_i)$ is either the identity or pseudo-Anosov, where $R_i$ is a path component of $\overline{S_g\setminus N}$ and $\eta_i:\map(R_i)\to\map(S_g)$ is the natural inclusion map. For $1\leq j\leq s$, $F_j$'s (or $\eta_j(F_j)$'s) will be called the \textit{canonical components} of $F$. The product $T_{c_1}^{q_1}T_{c_2}^{q_2}\cdots T_{c_{\ell}}^{q_{\ell}}$ appearing in (\ref{eqn:can_form}) will be called the \textit{multitwist component} of $F$. The decomposition of the form (\ref{eqn:can_form}) will be called the \textit{canonical decomposition (or the Nielsen decomposition)} of $F$. The positive integer $n$ will be called the \textit{degree} of $F$. For a multicurve $C$, the cut surface obtained by capping the boundary components of $\overline{S_g\setminus N}$ by marked disks will be denoted by $S_g(C)$, where $N$ is a closed regular neighborhood of $C$.

Suppose that $F,G \in \map(S_g)$ generate an infinite metacyclic subgroup of $\map(S_g)$ such that $\langle F \rangle \lhd \langle F,G \rangle$. Then it follows that $F$ and $G$ satisfy the relation $G^{-1}FG = F^k$, for some nonzero integer $k$. Hence, the group $\langle F,G \rangle$ is a semidirect product of $\langle F\rangle$ and $\langle G \rangle$, and will be denoted by $\langle F \rangle \rtimes_k \langle G \rangle$. In Section~\ref{sec:infinite_metacyclic}, we derive the main results of this paper. To begin with, in Subsection~\ref{subsec:infinite_metacyclic_pa}, we derive necessary and sufficient conditions for the existence of infinite metacyclic subgroups of $\map(S_g)$ with a pseudo-Anosov generator depending upon the Nielsen-Thurston type of the other generator (see Theorem \ref{thm:main_thm2}). We achieve this by analyzing its invariant foliations and the dilatation homomorphism (see~\cite{mccarthy}). In particular, we have given necessary and sufficient conditions under which a pseudo-Anosov mapping class $F$ forms an infinite metacyclic subgroup $\langle F,G \rangle$ with a nontrivial periodic mapping class $G$ such that $\langle F\rangle \lhd \langle F,G \rangle$. Furthermore, for other types of $G$, we have the following main result.

\begin{thm*}
\label{mainthm1}
For $g\geq 2$, consider nontrivial mapping classes $F,G \in \map(S_g)$. Let $\langle F,G \rangle$ be metacyclic with $\langle F\rangle \lhd \langle  F,G \rangle$. Then the following statements hold.
\begin{enumerate}[(i)]
\item If $F$ is a pseudo-Anosov, then $G$ cannot be an infinite order reducible mapping class.
\item If $F$ and $G$ are pseudo-Anosov, then $\langle F,G \rangle$ is abelian. Furthermore, either $\langle F,G\rangle\cong \Z$ or $\langle F,G\rangle \cong \mathbb{Z}_n\times \mathbb{Z}$ for some $n\in \N$.
\item Let $G$ be pseudo-Anosov and $\langle F,G\rangle$ is non-abelian. Then $F$ is a reducible mapping class of finite order. 
\end{enumerate}
\end{thm*}

In Subsection \ref{subsec:infinite_metacyclic_reducible}, by decomposing each reducible generator into its canonical components, we obtain necessary and sufficient conditions under which two reducible elements of $\map(S_g)$ form an infinite metacyclic subgroup. In this direction, we have our second main result (see Theorem \ref{thm:main_thm1}) which generalizes a result of Ivanov (see \cite[Theorem 7.5A]{ivanov2}).

\begin{thm*}
\label{mainthm2}
For $g\geq 2$, let $F,G \in \map(S_g)$ be two nontrivial mapping classes such that at least one of $F$ or $G$ is of infinite order and neither $F$ nor $G$ is pseudo-Anosov. Assume that $F,G$ have degrees $n,m$, with multitwist components $$T_{c_1}^{q_1}T_{c_2}^{q_2}\cdots T_{c_{\ell}}^{q_{\ell}} \text{ and } T_{c'_1}^{q'_1}T_{c'_2}^{q'_2}\cdots T_{c'_{\ell'}}^{q'_{\ell'}}, $$ respectively, where $q_i,q_i'\in \Z$, $\C(F)=\{c_1,c_2,\dots,c_{\ell}\}$, and $\C(G)=\{c'_1,c'_2,\dots,c'_{\ell'}\}$. Then $\langle F,G \rangle$ is an infinite metacyclic subgroup with $\langle F \rangle \lhd \langle F,G \rangle$ if and only if the following conditions hold.
\begin{enumerate}[(i)]
\item $\C(F)\cup \C(G)$ is a multicurve.
\item If $F$ is periodic with $G^{-1}FG=F^k$, then $k^m\equiv 1\pmod n$.
\item Define $A_i:=\{c_j\in \C(F)~|~q_j=q_i\}$, $B_i:=\{c_j\in \C(F)~|~q_j=kq_i\}$, and $C_i:=\{c'_j\in \C(G)~|~q'_j=q'_i\}$. Then $G(A_i)=B_i$, $G(B_i)=A_i$, and $F(C_i)=C_i$ for every $i$.
\item For every path component $R$ of $S_g(\mathcal{C}(F)\cup \mathcal{C}(G))$, then $G_r^{-1}F_rG_r=F_r^{k^{p_r}}$, where $G_r, F_r\in \map(R)$ are induced by $G,F$, respectively, and $p_r$ is the size of orbit of $R$ under $G$.
\item For two path components $R,S$ of $S_g(\mathcal{C}(F)\cup \mathcal{C}(G))$ such that $G(R)=S$, then $F_r^k$ is conjugate to $F_s$, where $F_r\in \map(R), F_s\in \map(S)$ are induced by $F$.
\end{enumerate}
\end{thm*}

\noindent The following result is a direct consequence of Theorem \ref{mainthm2}.
\begin{cor*}
For $g\geq 2$, let $F,G \in \map(S_g)$ be two nontrivial mapping classes such that at least one of $F$ or $G$ is of infinite order and neither $F$ nor $G$ is pseudo-Anosov. Let $\langle F,G\rangle$ be an infinite metacyclic subgroup of $\map(S_g)$ with $\langle F\rangle \lhd \langle F,G \rangle$. Then the following statements hold.
\begin{enumerate}[(i)]
\item $F$ and $G$ are reducible mapping classes.
\item If $F,G$ are of infinite order such that $G$ is of odd degree, then $\langle F,G \rangle$ is abelian.
\item If $G$ is of infinite order of degree $1$, then $\langle F,G \rangle$ is abelian.
\end{enumerate}
\end{cor*}

\noindent By applying our main theorems, we have shown that infinite metacyclic subgroups of $\map(S_g)$ are abundant. In general, we have established that $\map(S_g)$ has infinite metacyclic subgroups isomorphic to $\Z_n \rtimes_k \Z$, $\Z\rtimes_k \Z_n$, and $\Z\rtimes_k \Z$. We have constructed several explicit examples (see Section \ref{sec:infinite_metacyclic} - \ref{sec:application}) of such subgroups.

In Section~\ref{sec:application}, we derive several other applications of our main results. In Subsection \ref{subsec:level_m}, we obtain the following characterization of the infinite metacyclic subgroups of level $m$ subgroups $\map(S_g)[m]$ of $\map(S_g)$ for $m \geq 3$.

\begin{prop*}
For $g\geq 2$ and $m\geq 3$, let $F,G \in\map(S_g)[m]$ be two nontrivial mapping classes. Then $\langle F,G\rangle$ is metacyclic with $\langle F\rangle\lhd\langle F,G\rangle$ if and only if the following hold.
\begin{enumerate}[(i)]
\item $F$ and $G$ are infinite order reducible mapping classes that commute.
\item $\C(F)\cup \C(G)$ is a multicurve.
\item The nontrivial canonical components of $F$ and $G$ are pseudo-Anosov mapping classes.
\item The nontrivial canonical components of $F$ and $G$ with the same support generate a cyclic group.
\end{enumerate}
\end{prop*}

\noindent Moreover, when $g \geq 3$, we show the existence of non-abelian infinite metacyclic subgroups in $\map(S_g)[2]$. The following construction is motivated by a family of Penner-type pseudo-Anosov mapping classes described in~\cite{chris}.
\begin{cor*}
For $g \geq 3$, there is an infinite metacyclic subgroup of $\langle F,G \rangle <\map(S_g)[2]$ isomorphic to $\Z \rtimes_{-1}\Z_2$, where $F$ is a Penner-type pseudo-Anosov and $G$ is a hyperelliptic involution.
\end{cor*}

In Subsection~\ref{subsec:bounds}, we have derived bounds on the order of a nontrivial periodic generator of an infinite metacyclic subgroup of $\map(S_g)$ that are realized (see Propsosition \ref{prop:bounds}). In particular, we have the following result.

\begin{prop*}
For $g\geq 2$, let $F,G\in \map(S_g)$ be two nontrivial mapping classes such that $\langle F,G\rangle$ is an infinite metacyclic subgroup with $\langle F \rangle \lhd \langle F,G \rangle$.
\begin{enumerate}[(i)]
\item Let $F$ be a pseudo-Anosov mapping class and $G$ be a periodic mapping class.
\begin{enumerate}
\item If $\langle F,G\rangle$ is abelian, then $2\leq |G|\leq 2g$. 
\item If $\langle F,G\rangle$ is non-abelian, then $2\leq |G|\leq 4g$.  
\end{enumerate}
\item Let $F$ be a reducible mapping class of infinite order and $G$ be a periodic mapping class.
\begin{enumerate}
\item If $\langle F,G\rangle$ is abelian, then $2\leq |G|\leq 2g+2$.
\item If $\langle F,G\rangle$ is non-abelian, then $2\leq |G|\leq 2g$.
\end{enumerate}
\item If $F$ is periodic and $\langle F,G\rangle$ is non-abelian, then $3\leq |F|\leq 2g+2$.
\end{enumerate}
Moreover, all of the above bounds are realized.
\end{prop*}

In Subsection \ref{subsec:elements_type}, we describe pseudo-Anosovs in $\map(S_g)$ which can be written as a product of two nontrivial periodic mapping classes of the same order.

\begin{cor*}
Let $\langle F,G \rangle < \map(S_g)$ be a non-abelian infinite metacyclic subgroup with $\langle F \rangle \lhd \langle F,G \rangle$, where $F$ is a pseudo-Anosov and $G$ is nontrivial periodic. Then, for integers $i,j$ such that $i$ is odd and $j$ is even, $G^iF^j$ is conjugate to $G^i$. In particular, $GF^2$ is conjugate to $G$, and therefore, $F^2$ can be written as a product of two nontrivial periodic mapping classes of the same order.
\end{cor*}

\noindent  As a final application to our theory, in Subsection \ref{subsec:centralizer}, we analyze the centralizers of irreducible periodic mapping classes in $\map(S_g)$ (see Proposition \ref{prop:centralizer}). In particular, we have the following result.
\begin{cor*}
Let $F\in \map(S_g)$ be an irreducible periodic mapping class. Then the centralizer of $F$ in $\map(S_g)$ is either $\langle F \rangle$ or $\langle F\rangle \times \langle i\rangle$, where $i$ is an hypereliiptic involution.
\end{cor*}
   
\section{Preliminaries}
\label{sec:prelim}
For $g\geq 2$, let $S_g$ be the connected closed orientable surface of genus $g$. The \textit{mapping class group} of $S_g$ is the group of path components of $\homeo(S_g)$, and it will be denoted by $\map(S_g)$. The elements of $\map(S_g)$ are called \textit{mapping classes}. The Nielsen-Thurston classification \cite{thurston2} asserts that each mapping class in $\map(S_g)$ is either periodic, reducible, or pseudo-Anosov. Furthermore, a pseudo-Anosov mapping class is neither periodic nor reducible.

\subsection{Periodic mapping classes}
\label{subsec:periodic}
In view of the Nielsen-Kerckhoff theorem \cite{kerckhoff}, a periodic mapping class $F\in \map(S_g)$ of order $n$ has a representative $\F$ of the same order (known as a \textit{Nielsen representative}) which induces a $\Z_n$-action on $S_g$ via isometries. The \textit{corresponding orbifold} of $F$ is the quotient orbifold $\Orb_{F}:=S_g/\langle \F \rangle$ (see \cite[Chapter 13]{thurston1}), which is homeomorphic to $S_{g_0}$, where $g_0$ is the \textit{orbifold genus} of $\Orb_{F}$. The $\Z_n$-action induces a branched covering $p:S_g \to \Orb_{F}$ with $k$ branch points (or \textit{cone points}) $x_1, \dots, x_k$ in $\Orb_{F}$ of orders $n_1,\dots, n_k$, respectively. The \textit{order} of a cone point $x_i$ is the order of the stabilizer subgroup of any point in the preimage of $x_i$. From orbifold covering space theory, the branch covering $p:S_g\to \Orb_{F}$ corresponds to an exact sequence 
\begin{equation*}
1 \longrightarrow \pi_1(S_g) \xrightarrow{p_*} \pi_1^{orb}(\Orb_{F}) \xrightarrow{\phi} \Z_n \longrightarrow 1.
\end{equation*}

Moreover, $\pi_1^{orb}(\Orb_{F})$ is a Fuchsian group~\cite{katok,macbeath} that has the following presentation:
\begin{equation*}
\label{eqn:fuchsian_present}
\left\langle \alpha_1,\beta_1,\dots,\alpha_{g_0},\beta_{g_0},\gamma_1,\dots,\gamma_k\mid \gamma_1^{n_1}=\dots=\gamma_k^{n_k}=\prod_{i=1}^k\gamma_i\prod_{i=1}^{g_0} [\alpha_i,\beta_i]=1 \right\rangle.
\end{equation*}
The epimomorphism $\phi:\pi_1^{orb}(\Orb_F)\to \Z_n$ (classically known as a \textit{surface kernel map}) is order-preserving on torsion elements and is given by $\phi(\gamma_i)=\F^{(n/n_i)d_i}$, where $\gcd(d_i,n_i) =1$, for $1\leq i\leq k$. The tuple $(g_0;n_1,\dots,n_k)$ is called the \textit{signature} of the quotient orbifold $\mathcal{O}_F$ which we denote by $\Gamma(\Orb_F)$. Each cone point $x_i$ of order $n_i$ in $\mathcal{O}_F$ lifts under $p$ to an orbit of size $n/n_i$ on $S_g$ and the \textit{local rotation} induced by $\Z_n$-action in this orbit is given by $2 \pi d_i^{-1}/n_i$, where $\gcd(d_i,n_i)=1$. Thus, the orbit data of a cyclic action along with the structure of its corresponding orbifold can be compactly encoded as a tuple of integers.
\begin{defn}
\label{defn:data_set}
For $n\geq2$, $g_0\geq0$, and $0\leq r \leq n-1$, a \textit{cyclic data set of degree $n$}, denoted by $n(D)$, is a tuple of the form 
\begin{center}
$D=(n,g_0,r;(d_1,n_1),\dots,(d_k,n_k))$
\end{center}
with the following conditions.
\begin{enumerate}[(i)]
\item $r>0$ if and only if $k=0$, and when $r>0$, then $\gcd(r,n)=1$.
\item $n_i\geq2$,  $n_i\mid n$, $\gcd(d_i,n_i)=1$, $\text{ for all } i$.
\item $\lcm(n_1,\dots,\widehat{n_i},\dots,n_k)=\lcm(n_1,\dots,n_k)$, $\text{ for all } i$.
\item If $g_0=0$, then $\lcm(n_1,\dots,n_k)=n$.
\label{eqn:lcm}
\item $\sum_{i=1}^{k}\frac{n}{n_i}d_i\equiv 0 \pmod n $.
\label{eqn:angle_sum}
\item $\frac{2g-2}{n}=2g_0-2+\sum_{i=1}^{k}\left (1-\frac{1}{n_i}\right)$. \hfill \text{(Riemann-Hurwitz equation)}
\label{eqn:r-h}
\end{enumerate}
The number $g$ determined by the Riemann-Hurwitz equation is the \textit{genus} of the data set and will be denoted by $g(D)$.
\end{defn}

\noindent The quantity $r$ (in Definition \ref{defn:data_set}) will be nonzero if and only if $D$ represents a free rotation of $S_g$ by $2\pi r/n$. We will not include $r$ in the notation of a data set, whenever $r=0$. The significance of the cyclic data set is given in the following proposition due to Nielsen \cite{nielsen} (see also~\cite[Theorem 3.9]{rajeevsarathy5}).
\begin{prop}
Cyclic data sets of degree $n$ and genus $g$ are in one-to-one correspondence with conjugacy classes of periodic mapping classes of order $n$ in $\map(S_g)$.
\end{prop}

\noindent From here on, a periodic mapping class $F$ and its associated cyclic action $\F$ up to conjugacy will be represented by its corresponding data set, which we denote by $D_F$ and $D_{\F}$, respectively. The corresponding orbifold of $F$ will also be denoted by $\Orb_{\F}$.

We now state some results concerning nontrivial periodic mapping classes which will be used later. The following result due to Gilman \cite{gilman} characterizes irreducible periodic mapping classes $F\in \map(S_g)$ based on the corresponding orbifold $\Orb_F$.

\begin{theorem}
\label{thm:gilman}
For $g\geq 2$, let $F\in \map(S_g)$ be a nontrivial periodic mapping class. Then $F$ is irreducible if and only if $\Orb_F$ is a sphere with $3$ cone points.
\end{theorem}

\noindent We will now state a useful lemma \cite[Theorem 4.1]{2g+2} due to Kasahara.
\begin{lem}
\label{lem:reducible_bound}
For $g\geq 2$, let $F\in \map(S_g)$ be a nontrivial reducible periodic mapping class. Then $|F|\leq 2g+2$. The upper bound is realized if and only if $g$ is even and $\Gamma(\Orb_{F}) = (0;2,2,g+1,g+1)$. Furthermore, when $|F|< 2g+2$, we have $|F|\leq 2g$. Equivalently, if either $|F|=2g+1$ or $|F|>2g+2$, then $F$ is irreducible.
\end{lem}

\noindent Finally, we state the following assertion which follows from a result of Kulkarni \cite{kulkarni}.

\begin{lem}
\label{lem:4g+1}
There are no periodic mapping classes of order $4g+1$ in $\map(S_g)$.
\end{lem}

\subsection{Pseudo-periodic mapping classes}
\label{subsec:pp}
Let $F\in \map(S_g)$ be an infinite order reducible mapping class. From here on, we will use the notions of \textit{canonical decomposition} and the \textit{degree} of $F$ as defined in Section~\ref{sec:intro}. A mapping class is said to be \textit{pseudo-periodic} if it is either a nontrivial periodic or of infinite order reducible with no pseudo-Anosov canonical components in its canonical decomposition. Thus, a nontrivial periodic mapping class $F$ will be considered as a pseudo-periodic with $\C(F)=\emptyset$, degree $|F|$, and multitwist component equal to identity. We observe that multitwists are pseudo-periodic mapping classes having trivial periodic canonical components.

In the following example, we construct some infinite order pseudo-periodic mapping classes whose power is a Dehn twist about a simple closed curve.

\begin{exmp}
\label{exmp:compatibility}
Let $F\in \map(S_g)$ be pseudo-periodic mapping class such that $F^n=T_c$. Then $F$ is represented by an $\F \in \homeo(S_g)$ such that $\F(N)=N$, where $N$ is a closed annular neighborhood of $c$. Thus, $\F$ induces a $\Z_n$-action on $S_g(c)$ with two fixed points. Moreover, the sum of induced rotation angles about these fixed points is $2\pi /n$ modulo $2\pi$. Conversely, given nontrivial periodic mapping classes having a (two, in case $c$ is nonseparating) distinguished fixed point such that the sum of induced rotation angles about these fixed points is $2\pi /n$ modulo $2\pi$, one can reverse this process to recover $F$. (We refer the reader to~\cite{rajeevsarathy2, rajeevsarathy3,rajeevsarathy4, rajeevsarathy5} for details.) We illustrate this construction of roots of Dehn twists in Figure \ref{fig2}.
\begin{figure}[h]
\tiny
\begin{subfigure}{\textwidth}
\centering
\labellist
\pinlabel $(3,5)$ at 202, 85
\pinlabel $(3,5)$ at 202, 35
\pinlabel $D$ at 125, 47
\pinlabel $c$ at 285, 60
\endlabellist
\centering
\includegraphics[scale=0.52]{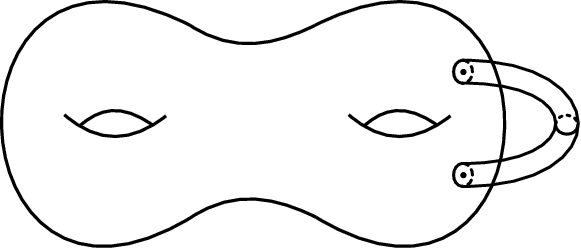}
\caption{The sum of local rotation angles about the two fixed points of cyclic action $D=(5,0;(4,5),(3,5),(3,5))$ associated with the pair $(3,5)$ is $-2\pi/5$ modulo $2\pi$. Hence, the action of $D$ can be extended to a pseudo-periodic mapping class $F\in\map(S_3)$ such that $F^5=T_c^{-1}$, where $c$ is a non-separating curve.}
\label{subfig1}
\end{subfigure}

\begin{subfigure}{\textwidth}
\labellist
\pinlabel $(7,8)$ at 220, 50
\pinlabel $(1,10)$ at 330, 45
\pinlabel $D_1$ at 120, 40
\pinlabel $D_2$ at 425, 40
\pinlabel $c$ at 275, 40
\endlabellist
\centering
\vspace{0.5cm}
\includegraphics[scale=0.52]{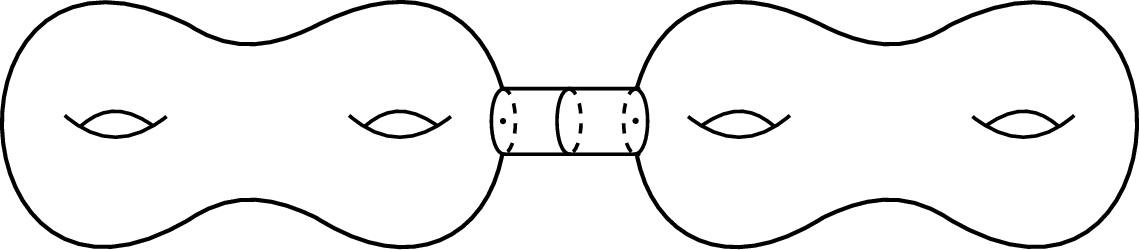}
\caption{The sum of local rotation angles about the fixed points of cyclic actions $D_1=(8,0;(1,2),(5,8),(7,8))$ and $D_2=(10,0;(1,2),(2,5),(1,10))$ associated with the pairs $(7,8)$ and $(1,10)$, respectively, is $-2\pi/40$ modulo $2\pi$. Since $\lcm(8,10)=40$, a pseudo-periodic $F\in \map(S_4)$ can be constructed from $D_1$ and $D_2$ such that $F^{40}=T_c^{-1}$, where $c$ is a separating curve.}
\label{subfig2}
\end{subfigure}
\caption{Construction of a pseudo-periodic mapping classes.}
\label{fig2}
\end{figure}
\end{exmp}

The angle sum condition in Example~\ref{exmp:compatibility} (in the construction of pseudo-periodic) generalizes to a formal ``compatibility condition" between pairs of orbits of one or more cyclic actions (see \cite{kapdi} for more details). 

\begin{defn}
For $i = 1,2$, let $O_i$ be an orbit of cyclic action $D_i$ such that $|O_1| = |O_2|$. Let $k$ be an integer such that $0\leq |k|\leq n/2$, where $n=\lcm(n(D_1),n(D_2))$.
\begin{enumerate}[(i)]
\item We say that $O_1$ and $O_2$ are \textit{trivially $n$-compatible} if $|O_1| = |O_2| = n$ (in this case $n(D_1)=n(D_2)$). 
\item Let the pair $(d_i,n_i)$ correspond to the orbit $O_i$ in the data set $D_i$, where we assume that $(d_i,n_i)=(0,1)$ if $|O_i|=n(D_i)$. We say that the orbits $O_1$ and $O_2$ are $|O_i|$-\textit{compatible with twist factor $k$} if
\begin{equation}
\label{eqn:twist_compa}
\frac{2\pi d_1^{-1}}{n_1}+\frac{2\pi d_2^{-1}}{n_2}\equiv \frac{2\pi k}{n} \pmod{2 \pi} .
\end{equation}
\end{enumerate}
\noindent When the twist factor associated with the compatibility of the $D_i$ is $0$, we simply say that the $D_i$ are $|O_i|$-\textit{compatible}.
\end{defn}

\subsection{Metacyclic groups}
\label{subsec:meta_group}
A group $H$ is said to be a \textit{metacyclic group} if there is a short exact sequence
\begin{equation}
\label{eqn:exact_seq}
1 \rightarrow N \rightarrow H \rightarrow L \rightarrow 1, 
\end{equation}
where $N$ and $L$ are cyclic groups. If a metacyclic group $H$ fits into an exact sequence as in (\ref{eqn:exact_seq}) that splits, then we say that $H$ is a \textit{split metacyclic group}. Thus, the split metacyclic group $H$ is isomorphic to the semidirect product $N \rtimes L$. Given integers $u,n \in \N$, a finite metacyclic group $H$ of order $u \cdot n$ admits the following presentation:
\begin{equation}
\label{eqn:presentation}
H=\langle \F,\G ~ | ~ \F^n = 1, \F^r = \G^u, \G^{-1}\F\G = \F^k\rangle,
\end{equation}
where $r \in \N$, $k \in \Z_n^{\times}$ such that $r \mid n$, $k^u \equiv 1 \pmod{n}$, and $r(k-1) \equiv 0 \pmod{n}$. For integers $m,n\in \N$ and $k\in \Z_n^{\times}$, a split metacyclic group admits the following presentation:
\begin{equation*}
H=\langle \F,\G ~ | ~ \F^n = 1, \G^m=1, \G^{-1}\F\G = \F^k\rangle \cong \Z_n\rtimes_k\Z_m.
\end{equation*}
Metacyclic groups have been completely classified by Hempel in~\cite{hempel}.

An \textit{infinite metacyclic group} is a metacyclic group of infinite order. It is known~\cite[Chapter 7]{hempel} that an infinite metacyclic group admits exactly one of the following presentations:
\begin{equation}
\begin{split}
\langle \F,\G ~ | ~ \G^{-1}\F\G = \F^k \rangle & \cong  \Z \rtimes_{k} \Z, \text{ for } k = \pm 1,\\
\langle \F,\G ~ | ~ \G^{2m} = 1, \G^{-1}\F\G = \F^k \rangle & \cong \Z \rtimes_{k} \Z_{2m}, \text{ for } k = - 1, m \in \N, \text{ and }\\
\langle \F,\G ~ | ~ \F^{n} = 1, \G^{-1}\F\G = \F^k \rangle & \cong  \Z_n \rtimes_{k} \Z, \text{ for } k \in \Z_n^{\times}, n \in \N.
\end{split}
\label{eqn:infinite_pres}
\end{equation}
Throughout this paper, we will only consider non-cyclic (i.e. two-generator) infinite metacyclic groups. As a consequence of the relation $\G^{-1}\F\G=\F^k$ in a metacyclic group $H=\langle \F,\G\rangle$, we have the following elementary lemma.
\begin{lem}
\label{lem:relation}
Let $H = \langle \F, \G \rangle$ be a metacyclic group, where $\G^{-1}\F\G = \F^k$. For integers $i,j$, we have:
\begin{enumerate}[(i)]
\item $\F^i \G^j = \G^{j}\F^{ik^j}$  and
\item $(\G^i \F^j)^{\ell} = \G^{i\ell} \F^{j(1+k^i+k^{2i}+\cdots + k^{i(\ell-2)}+k^{i(\ell-1)})}$.
\end{enumerate}
\end{lem}

\subsection{Induced orbifold automorphisms}
Let $\langle \F,\G\rangle$ be a metacyclic subgroup of $\homeo(S_g)$, where $\F$ has finite order. Each cone point $[x]\in \Orb_{\F}$ corresponds to a unique pair of the form $(c_x,n_x)$ in the data set $D_{\F}$ corresponding to $\F$. If $[x]\in \Orb_{\F}$ is not a cone point, then we take $(c_x,n_x)=(0,1)$. As $\langle \F \rangle \lhd H$, it is known~\cite{tucker} that $\G$ would induce a $\bar{\G} \in \homeo(\Orb_{\F})$ that preserves the set of cone points in $\Orb_{\F}$ along with their orders. We will call $\bar{\G}$, \textit{the induced automorphism on $\Orb_{\F}$ by $\G$,} and we formalize this notion in the following definition. 

\begin{defn}
\label{defn:ind_auto}
Let $\F\in \homeo(S_g)$ be a finite order map such that $|F|=n$. We say a $\bar{\G} \in \homeo(\Orb_{\F})$ is an \textit{automorphism of $\Orb_{\F}$} if for $[x],[y] \in \Orb_{\F}$, $k \in \Z_n^{\times}$, and $\bar{\G}([x]) = [y]$, we have 
\begin{enumerate}[(i)]
\item $n_x = n_y$, and 
\item $c_x = kc_y$. 
\end{enumerate} 
\end{defn}

\noindent We denote the group of automorphisms of $\Orb_{\F}$ by $\aut_k(\Orb_{\F})$. When $k=1$, we simply write $\aut(\Orb_{\F})$ instead of $\aut_1(\Orb_{\F})$. In the following lemma, we state some basic properties of induced automorphisms.

\begin{lem}[{\cite[Lemma 2.9]{sanghi1}}]
\label{lem:ind_auto}
Let $\F \in \homeo(S_g)$ be a map of order $n$ and $\G\in \homeo(S_g)$ be a map such that $\G^{-1} \F \G = \F^k$. Then $\G$ induces a $\bar{\G}\in \aut_k(\Orb_{\F})$ such that $$\Orb_{\F}/\langle \bar{\G}\rangle = S_g/\langle \F,\G\rangle.$$ Furthermore, $\G$ has infinite order if and only if $\bar{\G}$ has infinite order. If $|\G|=m$ then,
\begin{enumerate}[(i)]
\item $|\bar{\G}| \text{ divides } |\G|$, and 
\item $|\bar{\G}| < m$ if and only if $\F^{r}=\G^u$, for some $0< r<n$ and $0< u< m$.
\end{enumerate}
\end{lem}
\noindent We the refer the reader to~\cite{dhanwani,sanghi1} for further details on induced orbifold automorphisms.

\subsection{Pseudo-Anosov mapping classes}
\label{subsec:pa}
For $g\geq 2$, let $F\in \map(S_g)$ be \textit{pseudo-Anosov} mapping class. We will now describe a well-known construction of pseudo-Anosov mapping classes due to Penner \cite[Theorem 3.1]{penner}. 

\begin{theorem}
\label{thm:penner}
Let $A=\{a_1,a_2,\dots,a_m\}$ and $B=\{b_1,b_2,\dots, b_n\}$ be multicurves in $S_g$ such that $A \cup B$ fills $S_g$. Any product of positive powers of $T_{a_i}$ and negative powers of $T_{b_j}$, where each $a_i$ and each $b_j$ appear at least once, is a pseudo-Anosov mapping class.   
\end{theorem}

\noindent Note that a collection of simple closed curves $C$ in $S_g$ is said to \textit{fill} $S_g$ if $\overline{ S_g\setminus C}$ is a union of closed disks. Let $F,G\in \map(S_g)$ be nontrivial mapping classes, where $F$ is pseudo-Anosov with stretch factor $\lambda>1$ satisfying the relation $G^{-1}FG=F^k$. Let $(\f_s,\mu_s)$ and $(\f_u,\mu_u)$ be the stable and unstable singular measured foliations of $F$. We will require the  following result due to McCarthy \cite[Lemma 1]{mccarthy} in proving our results.
\begin{lem}
\label{lem:foliation}
Let $F,G \in \map(S_g)$ such that $F$ is pseudo-Anosov mapping class satisfying $G^{-1}FG=F^k$. Then $k=\pm 1$ and there exist a positive real number $\rho$ such that the following conditions hold. 
\begin{enumerate}[(i)]
\item if $k=1$, then $G(\f_s,\mu_s)=(\f_s,\rho^{-1} \mu_s)$ and $G(\f_u,\mu_u)=(\f_u,\rho \mu_u)$, 
\item if $k=-1$, then $G(\f_s,\mu_s)=(\f_u,\rho^{-1} \mu_u)$ and $G(\f_u,\mu_u)=(\f_s,\rho \mu_s)$.  
\end{enumerate}
\end{lem}

\begin{rem}
\label{rem:dil_homo}
Let $\h = \{G \in \map(S_g) : G(\f_s)=\f_s \text{ and } G(\f_u)=\f_u\}$ and $\R_+$ be the group of positive real numbers under multiplication. There is a homomorphism $\lambda:\h\to \R_+$ such that $\lambda(G) = \lambda_{G}$ with $G(\f_u,\mu_u)=(\f_u,\lambda_{G}\mu_u)$ and $G(\f_s,\mu_s)=(\f_s,\lambda_{G}^{-1}\mu_s)$ (see \cite{mccarthy}). This homomorphism is known as the \textit{dilatation homomorphism}.

For a singular point $p$ of $\f_u$, let $\h_p$ be the subgroup of $\h$ consisting of mapping classes that fix $p$. Let $\Le_p$ be the set of all singular leaves of $\f_u$ originating at the singular point $p$. The action of $\h_p$ on $\Le_p$ induces a homomorphism $\phi_p:\h_p\to \Sigma_{|\Le_p|}$, where $\Sigma_{|\Le_p|}$ is the permutation group on $|\Le_p|$ letters.
\end{rem}

\noindent The image and kernel of the dilatation homomorphism $\lambda$ have also been described in \cite[Lemmas 2-3]{mccarthy}.
\begin{lem}
\label{lem:dilat_homo}
For the dilatation homomorphism $\lambda:\mathcal{H}\to \mathbb{R}_+$, we have that $\lambda(\mathcal{H})$ is infinite cyclic and $\ker \lambda$ is a finite group. 
\end{lem}

\section{Infinite metacyclic subgroups of mapping class group}
\label{sec:infinite_metacyclic}
For $g\geq 2$ and two nontrivial periodic mapping classes $F,G\in \map(S_g)$, the necessary and sufficient number-theoretic conditions under which conjugates $F',G'$ (of $F,G$ resp.) generate a finite metacyclic group has been derived in \cite{dhanwani,sanghi1,sanghi2}. In this section, we analyze the infinite metacyclic subgroups of $\map(S_g)$. From here on, for $F,G\in \map(S_g)$, we will assume that if $\langle F,G\rangle$ is a metacyclic group, then $\langle F\rangle \lhd \langle F,G \rangle $, which implies that $G^{-1}FG = F^k$ for some nonzero integer $k$.

\subsection{Metacyclic subgroups with pseudo-Anosov generators}
\label{subsec:infinite_metacyclic_pa}
Let $F\in \map(S_g)$ be a pseudo-Anosov mapping class with stretch factor $\lambda>1$. Let $(\f_s,\mu_s)$ and $(\f_u,\mu_u)$ be the stable and unstable singular measured foliations for $F$, respectively. We will now prove our first main result concerning infinite metacyclic subgroups of $\map(S_g)$ with at least one pseudo-Anosov generator. The homomorphism $\phi_p$ in the statement of following theorem has been defined in Remark \ref{rem:dil_homo}.

\begin{theorem}[Main Theorem 1]
\label{thm:main_thm2}
For $g\geq 2$, consider nontrivial mapping classes $F,G \in \map(S_g)$.
\begin{enumerate}[(i)]
\item Let $\langle F,G \rangle$ be metacyclic with $\langle F\rangle \lhd \langle  F,G \rangle$. Then the following statements hold. 
\begin{enumerate}[(a)]
\item If $F$ is a pseudo-Anosov, then $G$ cannot be an infinite order reducible mapping class.
\item If $F$ and $G$ are pseudo-Anosov, then $\langle F,G \rangle$ is abelian. Furthermore, either $\langle F,G\rangle\cong \Z$ or $\langle F,G\rangle \cong \mathbb{Z}_n\times \mathbb{Z}$ for some $n\in \N$.
\item Let $G$ be pseudo-Anosov and $\langle F,G\rangle$ is non-abelian. Then $F$ is a reducible mapping class of finite order.
\end{enumerate}
\item Let $F$ be pseudo-Anosov and $G$ be either periodic or pseudo-Anosov. Then $\langle F,G \rangle$ is an abelian metacyclic subgroup if and only if
\begin{enumerate}
\item $G(\f_u,\mu_u)=(\f_u,\mu_u)$, $G(\f_u,\mu_u)=(\f_s,\mu_s)$, and
\item there exist a singular point $p$ of $\f_u$ such that $G^{-1}FGF^{-1}\in \ker \phi_p$.
\end{enumerate} 
\item Let $F$ be pseudo-Anosov and $G$ be periodic. Then $\langle F,G \rangle$ is non-abelian metacyclic subgroup with $\langle F\rangle \lhd \langle  F,G \rangle$ if and only if
\begin{enumerate}
\item $G(\f_u,\mu_u)=(\f_s,\mu_s)$, $G(\f_s,\mu_s)=(\f_u,\mu_u)$, and
\item there exist a singular point $p$ of $\f_u$ such that $G^{-1}FGF\in \ker \phi_p$.
\end{enumerate} 
\end{enumerate}
\end{theorem}

\begin{proof}
Let $F$ be a pseudo-Anosov mapping class with $(\f_u,\mu_u)$ and $(\f_s,\mu_s)$ as its unstable and stable invariant singular measured foliations, respectively.

To begin with, we consider the case when $\langle F,G\rangle$ is metacyclic with $\langle F\rangle\lhd\langle F,G\rangle$. Let $F$ be pseudo-Anosov and $G$ be an infinite order reducible mapping class. Consider the dilatation homomorphism $\lambda:\mathcal{H}\to \mathbb{R}_+$ (see Remark \ref{rem:dil_homo}). Since $G^{-1}FG=F^k$, where $k=\pm 1$, $G^2$ commutes with $F$. By Lemma \ref{lem:foliation}, $G^2\in \mathcal{H}$, and since $G$ is not pseudo-Anosov, $G^2\in \ker \lambda$. This is impossible since $\ker \lambda$ is finite and $G^2$ has infinite order.

Next, we consider the case when $F,G$ are pseudo-Anosov and $\langle F,G \rangle$ is a non-abelian metacyclic subgroup. If $G^{-1}FG=F^{-1}$, then $G^2$ commutes with $F$. By Lemma \ref{lem:foliation}, it follows that $G^2$ preserves $(\f_u,\mu_u)$ and $(\f_s,\mu_s)$. Thus, $F$ and $G$ keep $(\f_u,\mu_u)$ and $(\f_s,\mu_s)$ invariant, which contradicts Lemma \ref{lem:foliation}. Therefore, $\langle F,G \rangle$ is abelian, and from Lemma \ref{lem:foliation}, we have $\langle F,G\rangle\subset \mathcal{H}$. Since $\ker \lambda$ is a finite group, if $\ker \lambda \big|_{\langle F,G \rangle} \neq 1$, then $\langle F,G\rangle \cong \Z_n \times \Z$ for some $n\in \N$. Furthermore, if $\ker \lambda\big|_{\langle F,G \rangle} =1$, then $\langle F,G\rangle\cong \Z$.

Next, we assume that $G$ is a pseudo-Anosov mapping class and $\langle F,G \rangle$ is non-abelian metacyclic subgroup. As discussed above, $F$ can not be pseudo-Anosov. Let $F$ be an infinite order reducible mapping class. Since $\C(G^{-1}FG)=G^{-1}(\C(F))$, $\C(F^{-1})=\C(F)$ and $G^{-1}FG=F^{-1}$, it follows that $G(\C(F))=\C(F)$. But as $\C(F)\neq \emptyset$ and $G$ is irreducible, this contradicts our assumption. Now, assume that $F$ is periodic. If $F$ is irreducible, then, by Theorem \ref{thm:gilman}, $\Orb_F \approx S_{0,3}$. Since $GFG^{-1}=F^k$, $G$ induces an infinite order mapping class in $\map(S_{0,3})$ which is not possible. Hence, $F$ must be reducible periodic. 

Finally, we consider the case when $F$ is pseudo-Anosov and $G$ is either periodic or pseudo-Anosov. Consider the homomorphism $\phi_p$  defined in Remark \ref{rem:dil_homo}. From Lemma \ref{lem:foliation}, if $G^{-1}FG=F$, then $G(\f_u,\mu_u)=(\f_u,\mu_u)$ and $G(\f_s,\mu_s)=(\f_s,\mu_s)$. Furthermore, it is apparent that $G^{-1}FGF^{-1} \in \ker \phi_p$. Conversely, we assume that $G(\f_u,\mu_u)=(\f_u,\mu_u)$, $G(\f_s,\mu_s)=(\f_s,\mu_s)$, and $G^{-1}FGF^{-1}\in \ker \phi_p$ for some singular point $p$ of $\f_u$. Let $\F$ and $\G$ be representatives of $F$ and $G$, respectively. Since $G^{-1}FGF^{-1}\in \ker\phi_p$, we have  $\G^{-1}\F\G\F^{-1}(L)=L$, where $L$ is a leaf of $\f_u$ originating at the singular point $p$. Since $\lambda(\G^{-1}\F\G\F^{-1})=1$, $\G^{-1}\F \G \F^{-1}$ fixes $L$ pointwise. Since $L$ is dense in $S_g$ \cite[Theorem 9.6]{flp}, we must have $\G^{-1}\F \G \F^{-1}=1$, and hence $G^{-1}FG=F$. By a similar argument, (iii) follows.
\end{proof}
\noindent We address the case when $G$ is a pseudo-Anosov mapping class and $F$ is a nontrivial periodic mapping class in the following remark. 

\begin{rem}
For $g\geq 2$, let $F,G\in \map(S_g)$ be such that $G$ is a pseudo-Anosov mapping class and $F$ is a nontrivial periodic mapping class. By Birman-Hilden theory \cite{birman73,harvey75}, it follows that $\langle F,G \rangle$ is metacyclic with $\langle F\rangle \lhd \langle  F,G \rangle$ if and only if there exist a pseudo-Anosov mapping class $\bar{G}\in \map(\Orb_F)$ such that $\bar{G}$ lifts to $G$ under the branched cover $p:S_g\to \Orb_F$. By removing branch points and their preimages, $p$ can be considered an unbranched cover between punctured surfaces. Since $p$ is an abelian cover, an $\bar{G}$ lifts under $p$ if and only if the induced isomorphism $\bar{G}_{\#} \in \aut(H_1(\Orb_F,\Z))$ leaves the subgroup of $H_1(\Orb_F,\Z)$ corresponding to the cover $p$ invariant. This homological criterion is often straightforward to compute (see~\cite{ghaswala17,rajeevsarathy211}). 
\end{rem}

Now, we construct several infinite metacyclic subgroups of $\map(S_g)$ with a pseudo-Anosov generator. In the following example, we describe a non-abelian infinite metacyclic subgroup having a nontrivial periodic and a pseudo-Anosov generator.

\begin{exmp}
\label{exmp:z-1z4s4g}
For $g\geq 1$, let $G$ be a rotation of $S_{4g}$ by $2\pi/4$ as shown in the Figure \ref{fig:z-1z4s4g}. By considering the multicurves $A=\{a_1,a_2,\dots, a_{4g+1}\}$ and $B=\{b_1,b_2,\dots, b_{4g+1}\}$, we see that the curves in $A \cup B$ fill $S_{4g}$. From Theorem \ref{thm:penner}, it follows that 
\begin{center}
$F=\prod_{i=1}^{4g+1} T_{a_i} \prod_{j=1}^{4g+1} T_{b_j}^{-1}$
\end{center}
is a pseudo-Anosov mapping class. For $1\leq i\leq 3g$ and $1 \leq i' \leq g$, we have $G(a_i)=b_{g+i}$, $G(b_i)=a_{g+i}$, $G(a_{3g+i'})=b_{i'}$, $G(b_{3g+i'})=a_{i'}$, and $G$ exchanges the curves $a_{4g+1}$ and $b_{4g+1}$. Therefore, $G^{-1}FG=F^{-1}$, and so we have $\langle F,G\rangle\cong \mathbb{Z}\rtimes_{-1}\mathbb{Z}_4$. We observe that $\langle F, G^2 \rangle \cong \Z \times \Z_2$.
\begin{figure}[H]
\labellist
\tiny

\pinlabel $a_1$ at -15, 285
\pinlabel $b_1$ at 50, 235
\pinlabel $a_2$ at 95, 305
\pinlabel $a_g$ at 132, 305
\pinlabel $b_g$ at 160, 248

\pinlabel $a_{2g}$ at 360, 175
\pinlabel $b_{2g}$ at 210, 145
\pinlabel $b_{g+2}$ at 360, 75
\pinlabel $a_{g+1}$ at 235, 50
\pinlabel $b_{g+1}$ at 280, -10

\pinlabel $b_{3g}$ at 400, 250
\pinlabel $a_{3g}$ at 445, 305
\pinlabel $a_{2g+2}$ at 480, 258
\pinlabel $b_{2g+1}$ at 520, 235
\pinlabel $a_{2g+1}$ at 590, 280

\pinlabel $b_{3g+1}$ at 285, 580
\pinlabel $a_{3g+1}$ at 360, 520
\pinlabel $b_{3g+2}$ at 210, 485
\pinlabel $b_{4g}$ at 350, 420
\pinlabel $a_{4g}$ at 225, 390

\pinlabel $a_{4g+1}$ at 340, 275
\pinlabel $b_{4g+1}$ at 230, 320

\pinlabel $\G$ at 450, 430
\pinlabel $\frac{2\pi }{4}$ at 460, 390

\endlabellist
\centering
\includegraphics[scale=0.4]{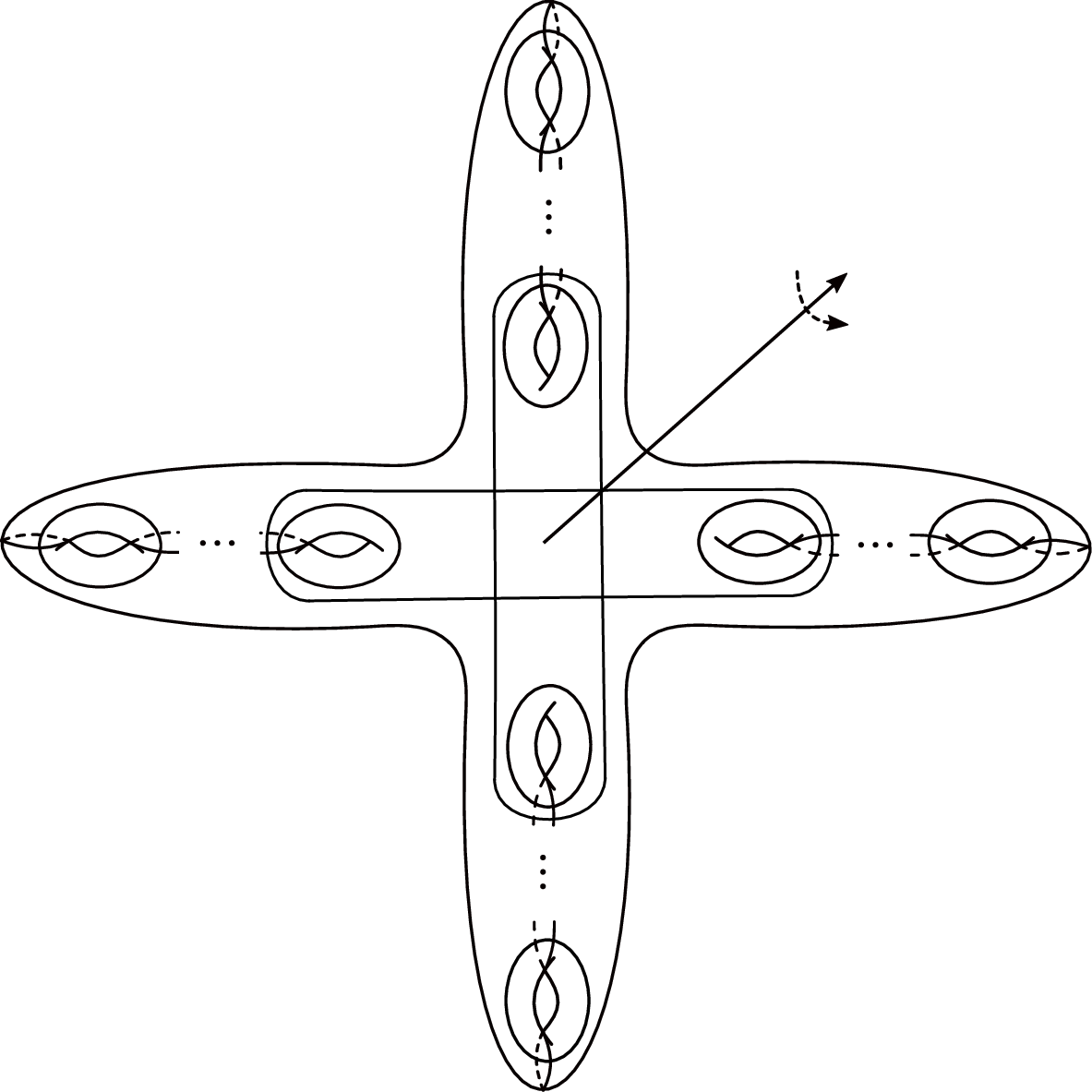}
\caption{Realization of an infinite metacyclic subgroup $\langle F,G\rangle < \map(S_{4g})$ isomorphic to $\mathbb{Z}\rtimes_{-1} \mathbb{Z}_4$ generated by a periodic mapping class $G$ of order $4$ and a pseudo-Anosov mapping class $F$.}
\label{fig:z-1z4s4g}
\end{figure}
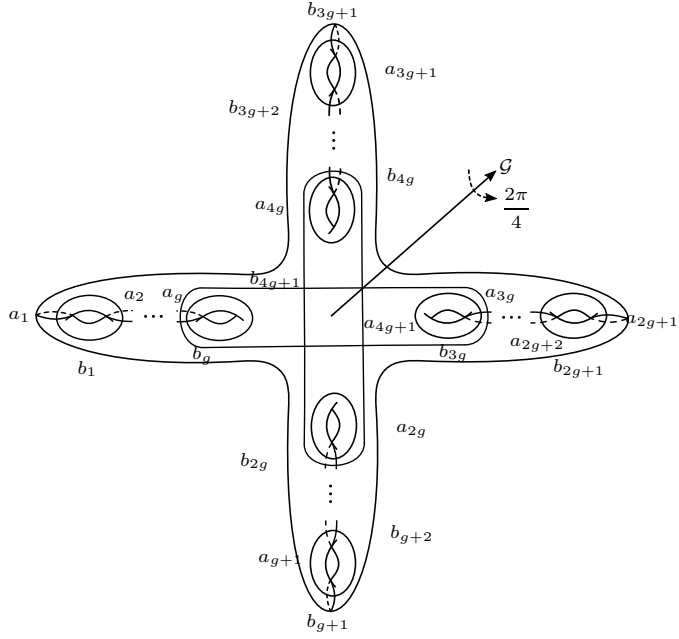
\end{exmp}

Note that the construction described in Example \ref{exmp:z-1z4s4g} generalizes to any even integer $n\geq 4$, where $n$ is the order of the periodic generator. For even genera, we will now provide an example of a metacyclic subgroup of $\map(S_g)$ isomorphic to $\Z \rtimes_{-1} \Z_2$. 

\begin{exmp}
\label{exmp:lower_bound_pa}
For an integer $h\geq 1$, let $A=\{a_1,a_2,\dots,a_{2h},d\}$ and $B=\{b_1,b_2,\dots,b_{2h},c\}$ be two multicurves in $S_{2h}$,  and $G\in \map(S_{2h})$ be an involution, as shown in the Figure \ref{fig:z-1z2sg}. Since $A\cup B$ fills $S_{2h}$, by Theorem \ref{thm:penner}, the mapping class $$F= \prod_{i=1}^{2h} T_{a_i}T_d \prod_{j=1}^{2h}T_{b_j}^{-1}T_c^{-1}$$ is pseudo-Anosov. For $1\leq i\leq 2h$, $G$ maps $a_i$ to $b_{2h+1-i}$, $b_i$ to $a_{2h+1-i}$, and $c$ to $d$. Therefore, we have $G^{-1}FG=F^{-1}$, and so it follows that $\langle F,G \rangle <\map(S_{2h})$ is isomorphic to $\mathbb{Z}\rtimes_{-1} \mathbb{Z}_2$.
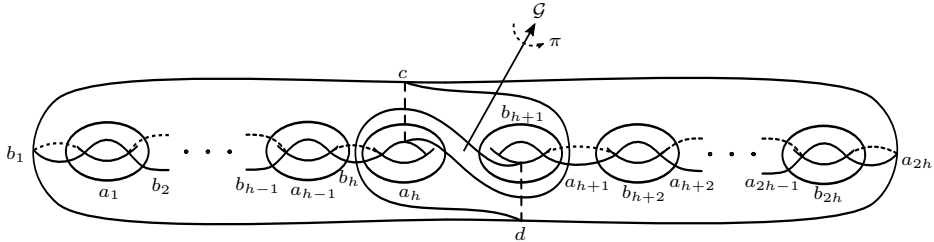
\begin{figure}[H]
\labellist
\tiny

\pinlabel $b_1$ at -10, 45
\pinlabel $b_2$ at 80, 25
\pinlabel $b_{h-1}$ at 140, 25
\pinlabel $b_h$ at 197, 30
\pinlabel $b_{h+1}$ at 305, 70
\pinlabel $b_{h+2}$ at 380, 20
\pinlabel $b_{2h}$ at 495, 20

\pinlabel $a_1$ at 48, 20
\pinlabel $a_{h-1}$ at 175, 20
\pinlabel $a_h$ at 235, 20
\pinlabel $a_{h+1}$ at 345, 25
\pinlabel $a_{h+2}$ at 410, 25
\pinlabel $a_{2h-1}$ at 460, 25
\pinlabel $a_{2h}$ at 550, 40

\pinlabel $c$ at 230, 95
\pinlabel $d$ at 303, -5

\pinlabel $\pi$ at 325, 115
\pinlabel $\G$ at 315, 135

\endlabellist
\centering
\includegraphics[scale=0.6]{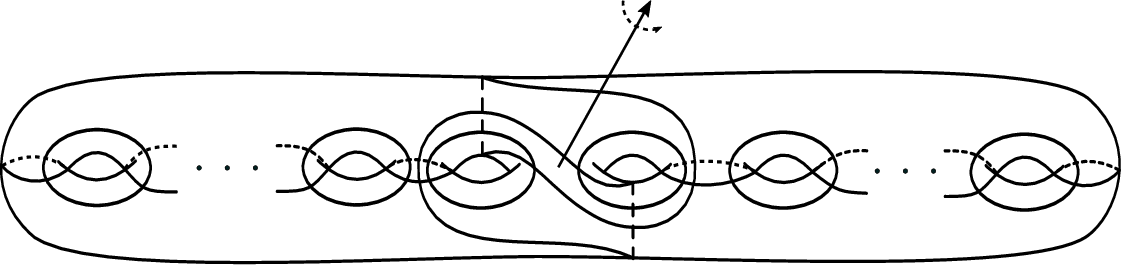}
\caption{An infinite metacyclic subgroup $\langle F,G\rangle<\map(S_{2h})$ isomorphic to $\mathbb{Z}\rtimes_{-1} \mathbb{Z}_2$, generated by an involution $G$ and a pseudo-Anosov $F$.}
\label{fig:z-1z2sg}
\end{figure}
\end{exmp}

\noindent Examples \ref{exmp:z-1z4s4g} - \ref{exmp:lower_bound_pa} together generalize to the following corollary. 

\begin{cor}
For an even positive integer $n\mid g$, there is an infinite metacyclic subgroup of $\langle F,G \rangle <\map(S_g)$ isomorphic to $\Z \rtimes_{-1}\Z_n$, where $F$ is a Penner-type pseudo-Anosov and $G$ is  periodic with $D_G=(n,g/n;(1,n),(n-1,n))$.
\end{cor}

For positive integer $m$, let $\langle F,G \mid G^{2m}=1,G^{-1}FG=F^{-1}\rangle$ be an infinite metacyclic subgroup of $\map(S_g)$, where $F$ is a pseudo-Anosov mapping class and $G$ is a periodic mapping class. Then it is easily seen that $\langle F, G^2\rangle$ is abelian (as in Example \ref{exmp:z-1z4s4g}). However, a natural question is whether every infinite abelian metacyclic subgroup of $\map(S_g)$ arises this way. The following example shows that this is not true in general.

\begin{exmp}
\label{exmp:zz4s4g+1}
For $g\geq 1$, let $G$ be a free rotation of $S_{4g+1}$ by $2\pi/4$ as shown in the Figure \ref{fig:zz4s4g+1}. We observe that the multicurves $A=\{a_1,a_2,\dots, a_{4g+1}\}$ and $B=\{b_1,b_2,\dots, b_{4g+4}\}$ fill $S_{4g+1}$.
\begin{figure}[H]
\labellist
\tiny

\pinlabel $b_1$ at -15, 285
\pinlabel $a_1$ at 50, 235
\pinlabel $b_2$ at 95, 305
\pinlabel $b_g$ at 135, 305
\pinlabel $a_g$ at 175, 252

\pinlabel $a_{2g}$ at 355, 175
\pinlabel $b_{2g}$ at 215, 145
\pinlabel $b_{g+2}$ at 355, 75
\pinlabel $a_{g+1}$ at 225, 50
\pinlabel $b_{g+1}$ at 280, -10

\pinlabel $a_{3g}$ at 395, 255
\pinlabel $b_{3g}$ at 435, 305
\pinlabel $b_{2g+2}$ at 480, 258
\pinlabel $a_{2g+1}$ at 535, 235
\pinlabel $b_{2g+1}$ at 590, 285

\pinlabel $b_{3g+1}$ at 285, 580
\pinlabel $a_{3g+1}$ at 350, 520
\pinlabel $b_{3g+2}$ at 210, 485
\pinlabel $b_{4g}$ at 350, 420
\pinlabel $a_{4g}$ at 225, 390

\pinlabel $b_{4g+1}$ at 235, 303
\pinlabel $b_{4g+2}$ at 315, 240
\pinlabel $b_{4g+3}$ at 338, 270
\pinlabel $b_{4g+4}$ at 257, 330
\pinlabel $a_{4g+1}$ at 250, 260

\pinlabel $\G$ at 450, 430
\pinlabel $\frac{2\pi }{4}$ at 460, 390

\endlabellist
\centering
\includegraphics[scale=0.4]{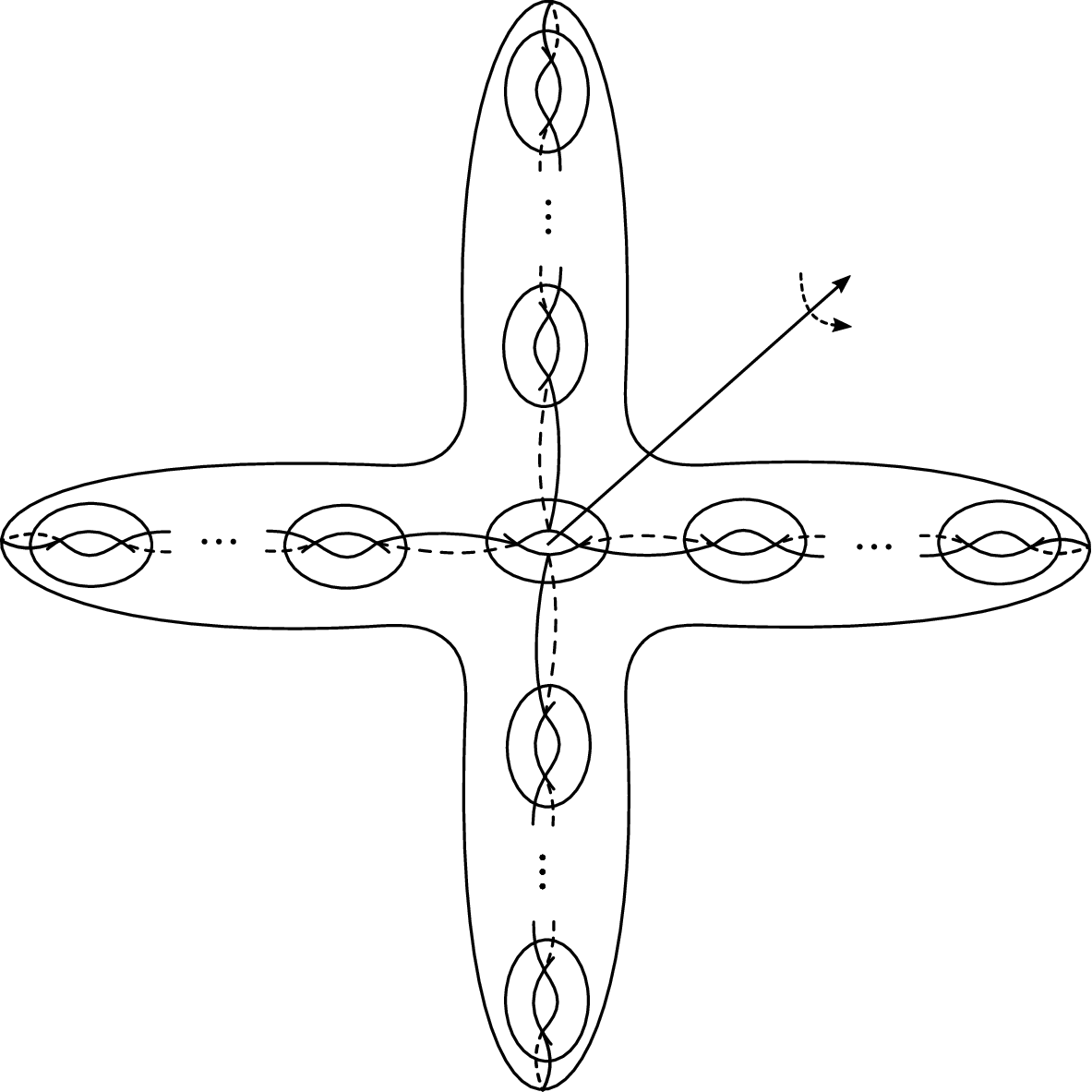}
\caption{Realization of an infinite metacyclic subgroup $\langle F,G\rangle < \map(S_{4g+1})$ isomorphic to $\mathbb{Z}\times \mathbb{Z}_4$ generated by a periodic mapping class $G$ of order $4$ and a pseudo-Anosov mapping class $F$.}
\label{fig:zz4s4g+1}
\end{figure}
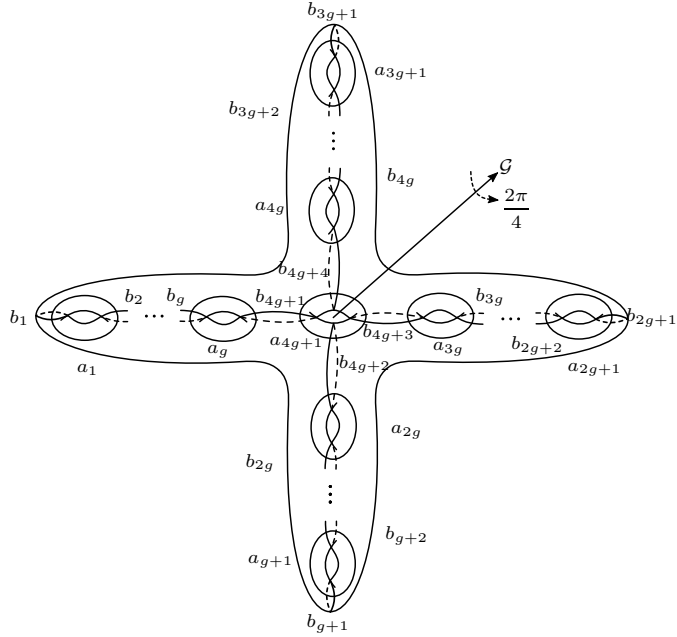
\noindent From the Theorem \ref{thm:penner}, the mapping class $$F=\prod_{i=1}^{4g+1} T_{a_i} \prod_{j=1}^{4g+4} T_{b_j}^{-1}$$ is pseudo-Anosov. For $1\leq i\leq 3g$, $1\leq i'\leq g$, and $1\leq j\leq 3$, we have that $G(a_i)=a_{i+g}$, $G(b_i) = b_{i+g}$, $G(a_{3g+i'})=a_{i'}$, $G(b_{3g+i'})=b_{i'}$, $G(a_{4g+1})=a_{4g+1}$, $G(b_{4g+j})=b_{4g+j+1}$, and $G(b_{4g+4})=b_{4g+1}$. By construction, we have $GF=FG$, and so it follows that $\langle F,G\rangle\cong \mathbb{Z}\times \mathbb{Z}_4$.

Assume that $\langle F, G\rangle$ is a subgroup of a non-abelian infinite metacyclic subgroup $\langle F, G'\rangle$, where $(G')^2 = G$. It follows from \cite[Corollary 5.7]{dhanwani} that $G$ is primitive. Therefore, such a $G'$ cannot exist.
\end{exmp}

In Example \ref{exmp:zz4s4g+1}, the periodic generator was represented by a nontrivial free rotation, but in the following example, the periodic generator is represented by a nontrivial non-free rotation.

\begin{exmp}
\label{exmp:zz4s4g}
For $g\geq 1$, let $G$ be a rotation of $S_{4g}$ by $2\pi/4$ as shown in the Figure \ref{fig:zz4s4g}. We observe that the multicurves $A=\{a_1,a_2,\dots, a_{4g+1}\}$ and $B=\{b_1,b_2,\dots, b_{4g}\}$ fill $S_{4g}$. From the Theorem \ref{thm:penner}, the mapping class $$F=\prod_{i=1}^{4g+1} T_{a_i} \prod_{j=1}^{4g} T_{b_j}^{-1}$$ is pseudo-Anosov. For $1\leq i\leq 3g$, $1\leq i'\leq g$, and $1\leq j\leq 3$, we have that $G(a_i)=a_{i+g}$, $G(b_i) = b_{i+g}$, $G(a_{3g+i'})=a_{i'}$, $G(b_{3g+i'})=b_{i'}$, and $G(a_{4g+1})=a_{4g+1}$. By construction, we have $GF=FG$, and so it follows that $\langle F,G\rangle\cong \mathbb{Z}\times \mathbb{Z}_4$.
\begin{figure}[H]
\labellist
\tiny

\pinlabel $b_1$ at -15, 285
\pinlabel $a_1$ at 50, 235
\pinlabel $b_2$ at 95, 305
\pinlabel $b_g$ at 135, 305
\pinlabel $a_g$ at 230, 280

\pinlabel $a_{2g}$ at 290, 220
\pinlabel $b_{2g}$ at 215, 145
\pinlabel $b_{g+2}$ at 355, 75
\pinlabel $a_{g+1}$ at 225, 50
\pinlabel $b_{g+1}$ at 280, -10

\pinlabel $a_{3g}$ at 340, 280
\pinlabel $b_{3g}$ at 435, 305
\pinlabel $b_{2g+2}$ at 480, 258
\pinlabel $a_{2g+1}$ at 535, 235
\pinlabel $b_{2g+1}$ at 590, 285

\pinlabel $b_{3g+1}$ at 285, 580
\pinlabel $a_{3g+1}$ at 350, 520
\pinlabel $b_{3g+2}$ at 210, 485
\pinlabel $b_{4g}$ at 350, 420
\pinlabel $a_{4g}$ at 290, 345

\pinlabel $a_{4g+1}$ at 270, 320

\pinlabel $\G$ at 450, 430
\pinlabel $\frac{2\pi }{4}$ at 460, 390

\endlabellist
\centering
\includegraphics[scale=0.4]{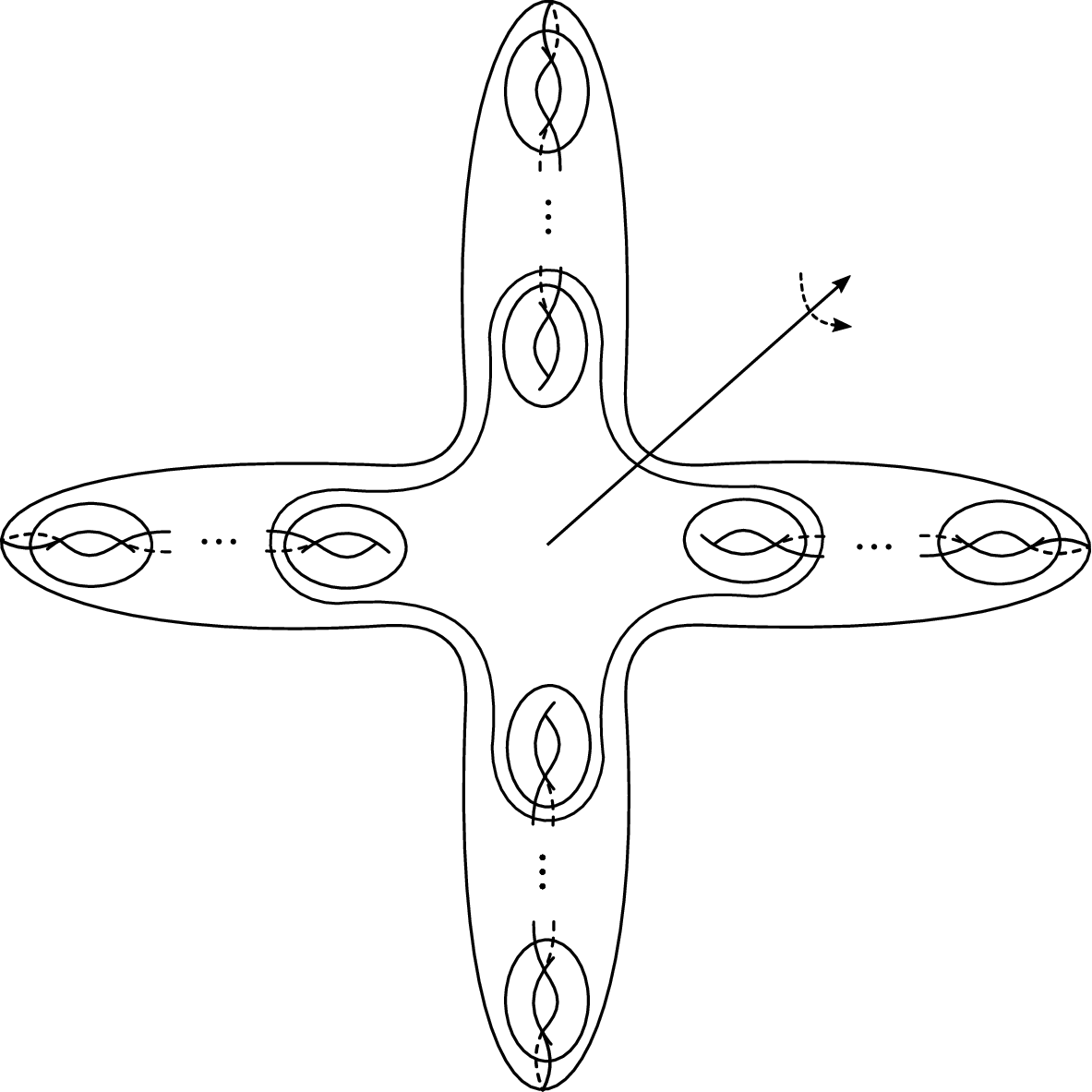}
\caption{Realization of an infinite metacyclic subgroup $\langle F,G\rangle < \map(S_{4g})$ isomorphic to $\mathbb{Z}\times \mathbb{Z}_4$ generated by a periodic mapping class $G$ of order $4$ and a pseudo-Anosov mapping class $F$.}
\label{fig:zz4s4g}
\end{figure}
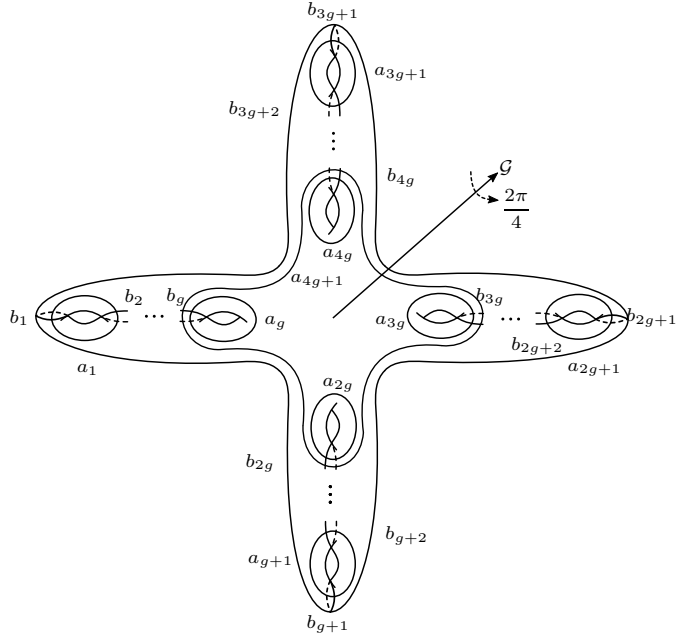
\end{exmp}

\noindent The Examples \ref{exmp:zz4s4g+1} - \ref{exmp:zz4s4g} can be generalized to the following. 

\begin{cor}
For any positive integer $n\geq 2$ such that $n\mid g$ (resp. $n \mid g-1$), there is an infinite metacyclic subgroup of $\langle F,G \rangle <\map(S_g)$ isomorphic to $\Z \times\Z_n$, where $F$ is a Penner-type pseudo-Anosov and $G$ is periodic with $D_G=(n,g/n;(1,n),(n-1,n))$ $($resp. $D_G=(n,(g+n-1)/n,1;-)$ $)$.
\end{cor}

\begin{rem}
\label{rem:non_penner_type}
An infinite metacyclic subgroup of $\langle F, G \rangle < \map(S_g)$ generated by a Penner-type pseudo-Anosov mapping class $F$ can also have pseudo-Anosovs of non-Penner-type. In Examples \ref{exmp:z-1z4s4g} and \ref{exmp:zz4s4g+1}, it can be seen that every pseudo-Anosov mapping class in $\langle F, G \rangle$ which is not a power of $F$ is a non-Penner-type pseudo-Anosov. In fact, $F$ can  also be replaced with a non-Penner-type pseudo-Anosov generator. In Example \ref{exmp:zz4s4g+1}, for $i\not \equiv 0 \pmod {|G|}$ and $j\neq 0$, the mapping class $G^iF^j$ is a non-Penner-type pseudo-Anosov, while in Example \ref{exmp:z-1z4s4g}, $G^iF^j$, where $i$ is an even positive integer such that $i\not \equiv 0 \pmod {|G|}$ and $j\neq 0$, is a non-Penner-type pseudo-Anosov. Furthermore, in each case, taking $j = \pm 1$ would yield elements that are possible generators of $\langle F,G \rangle$ in place of $F$.
\end{rem}

\begin{rem}
\label{rem:stretch_factor}
Let $F$ be a pseudo-Anosov generator of a metacyclic subgroup of $\map(S_g)$. Then there is no upper bound on the stretch factor $\lambda(F)$ of $F$. This follows from the simple fact that if  $\langle F,G \rangle$ is a metacyclic subgroup of $\map(S_g)$, then $\langle F^n,G \rangle$ is also a metacyclic subgroup for all $n>1$, where $\lambda(F^n) = \lambda(F)^n$.
\end{rem}

\subsection{Metacyclic subgroups with reducible generators of infinite order}
\label{subsec:infinite_metacyclic_reducible}
We begin this subsection with the following lemma which provide necessary and sufficient conditions under which two multitwists are equal.

\begin{lem}[{\cite[Lemma 3.17]{primer}}]
\label{lem:equal_multitwist}
Let $A=\{a_1,\dots,a_n\}$ and $B=\{b_1,\dots,b_m\}$ be two multicurves in $S_g$. Let $p_i$ and $q_i$ be nonzero integers. If
\begin{center}
$T_{a_1}^{p_1}\cdots T_{a_n}^{p_n}=T_{b_1}^{q_1}\cdots T_{b_m}^{q_m}$
\end{center}
in $\map(S_g)$, then $m=n$ and the sets $\{T_{a_i}^{p_i}\}$, $\{T_{b_i}^{q_i}\}$ are equal.
\end{lem}

\noindent We will now establish our second main result that gives necessary and sufficient conditions under which two mapping classes which are not pseudo-Anosov form an infinite metacyclic subgroup of $\map(S_g)$.

\begin{theorem}[Main Theorem 2]
\label{thm:main_thm1}
For $g\geq 2$, let $F,G \in \map(S_g)$ be two nontrivial mapping classes such that at least one of $F$ or $G$ is of infinite order and neither $F$ nor $G$ is pseudo-Anosov. Assume that $F,G$ have degrees $n,m$, with multitwist components $$T_{c_1}^{q_1}T_{c_2}^{q_2}\cdots T_{c_{\ell}}^{q_{\ell}} \text{ and } T_{c'_1}^{q'_1}T_{c'_2}^{q'_2}\cdots T_{c'_{\ell'}}^{q'_{\ell'}}, $$ respectively, where $q_i,q_i'\in \Z$, $\C(F)=\{c_1,c_2,\dots,c_{\ell}\}$, and $\C(G)=\{c'_1,c'_2,\dots,c'_{\ell'}\}$. Then $\langle F,G \rangle$ is an infinite metacyclic subgroup with $\langle F \rangle\lhd\langle F,G \rangle$ if and only if the following conditions hold.
\begin{enumerate}[(i)]
\item $\C(F)\cup \C(G)$ is a multicurve.
\item If $F$ is periodic with $G^{-1}FG=F^k$, then $k^m\equiv 1\pmod n$.
\item Define $A_i:=\{c_j\in \C(F)~|~q_j=q_i\}$, $B_i:=\{c_j\in \C(F)~|~q_j=kq_i\}$, and $C_i:=\{c'_j\in \C(G)~|~q'_j=q'_i\}$. Then $G(A_i)=B_i$, $G(B_i)=A_i$, and $F(C_i)=C_i$ for every $i$.
\item For every path component $R$ of $S_g(\mathcal{C}(F)\cup \mathcal{C}(G))$, $G_r^{-1}F_rG_r=F_r^{k^{p_r}}$, where $G_r, F_r\in \map(R)$ are induced by $G,F$, respectively, and $p_r$ is the size of orbit of $R$ under $G$.
\item For two path components $R,S$ of $S_g(\mathcal{C}(F)\cup \mathcal{C}(G))$ such that $G(R)=S$, $F_r^k$ is conjugate to $F_s$, where $F_r\in \map(R), F_s\in \map(S)$ are induced by $F$.
\end{enumerate} 
\end{theorem}
\begin{proof}
Let $\langle F,G \rangle$ be an infinite metacyclic subgroup of $\map(S_g)$. First, we assume that $F$ has infinite order. Since $G^{-1}FG = F^k$, where $k=\pm 1$, we have $G^{-1}F^nG = F^{kn}$, and so their multitwist components are equal, that is, $$T_{G^{-1}(c_1)}^{q_1}T_{G^{-1}(c_2)}^{q_2}\cdots T_{G^{-1}(c_{\ell})}^{q_{\ell}} = T_{c_1}^{kq_1}T_{c_2}^{kq_2}\cdots T_{c_{\ell}}^{kq_{\ell}}.$$ By Lemma \ref{lem:equal_multitwist}, it follows that $$\{T_{G^{-1}(c_i)}^{q_i} ~|~ 1 \leq i \leq \ell\} = \{T_{c_j}^{kq_j} ~|~ 1 \leq j \leq \ell\},$$ and so $G(A_i) = B_i$ and $G(B_i)=A_i$ for every $i$. Hence, $G(\C(F))=\C(F)$. Since $k=\pm 1$, $G^2$ commutes with $F$. By comparing the multitwist components in $FG^2F^{-1}=G^2$, we have $F(C_i)=C_i$ for every $i$. As $\C(G)$ is the intersection of all maximal reduction system of $G$, $\C(G)$ is contained in the maximal reduction system of $G$ containing $\C(F)$. Therefore, it follows that $\C(F)\cup \C(G)$ is a multicurve. The same conclusion holds trivially for the case when $F$ is periodic.  

Suppose that $G$ has infinite order and $F$ is periodic. Since $G^{-1}FG=F^k$, we have $G^{-ma}FG^{ma}=F^{k^{ma}}=F$, where $a=|k|$.  By comparing the multitwist components in $FG^{ma}F^{-1}=G^{ma}$, it follows that $$T_{F(c'_1)}^{aq'_1}T_{F(c'_2)}^{aq'_2}\cdots T_{F(c'_{\ell'})}^{aq'_{\ell'}}=T_{c'_1}^{aq'_1}T_{c'_2}^{aq'_2}\cdots T_{c'_{\ell'}}^{aq'_{\ell'}}.$$ By Lemma \ref{lem:equal_multitwist}, we have $F(C_i) = C_i$ for each $i$, and so $F(\C(G))=\C(G)$. Since $F(\C(G))=\C(G)$ and $G^{-m}(FG^mF^{-1})=F^{k^m-1}$, it follows that $F^{k^m-1}=1$. Therefore, $k^m\equiv 1\pmod n$, and we have established $(i)-(iii)$.

Restricting the relation $G^{-1}FG=F^k$ to a path component $R$ of $S_g(\C(F)\cup \C(G))$ gives $(G^{-p_r}FG^{p_r})\big |_R =F^{k^{p_r}}\big|_R$, where $p_r$ is the size of the orbit of $R$ under $G$. Therefore, $G_r^{-1}F_rG_r=F_r^{k^{p_r}}$, where $G_r,F_r\in \map(R)$ are induced by $G,F$, respectively. For two distinct path components $R,S$ of $S_g(\C(F)\cup \C(G))$ such that $G(R)=S$, restricting the relation $G^{-1}FG=F^k$ to $R$, it follows that $F_s$ is conjugate to $F_r^k$, where $F_r\in \map(R), F_s\in \map(S)$ are induced by $F$. This completes the argument for $(iv)-(v)$.
	
Conversely, we assume that $F$ and $G$ satisfies $(i)-(v)$. Since the relations $G_r^{-1}F_rG_r=F_r^{k^{p_r}}$ holds in $\map(R)$ for every path component $R$ of $S_g(\C(F)\cup \C(G))$, it follows from conditions $(i)-(iii),\ (v)$ that the relation $G^{-1}FG=F^k$ holds in $\map(S_g)$. Hence, $\langle F,G \rangle$ is an infinite metacyclic subgroup.
\end{proof}

\noindent We have the following direct consequence of Theorem \ref{thm:main_thm1}.

\begin{cor}
\label{cor:main}
For $g\geq 2$, let $F,G \in \map(S_g)$ be two nontrivial mapping classes such that at least one of $F$ or $G$ is of infinite order and neither $F$ nor $G$ is pseudo-Anosov. Let $\langle F,G\rangle$ be an infinite metacyclic subgroup of $\map(S_g)$ with $\langle F\rangle \lhd \langle F,G \rangle$. Then the following statements hold.
\begin{enumerate}[(i)]
\item $F$ and $G$ are reducible mapping classes.
\item If $F,G$ are of infinite order such that $G$ is of odd degree, then $\langle F,G \rangle$ is abelian.
\item If $G$ is of infinite order of degree $1$, then $\langle F,G \rangle$ is abelian.
\end{enumerate}
\end{cor}
\begin{proof}
By Theorem \ref{thm:main_thm1} (i),(iii), $F$ and $G$ preserve the multicurve $\C(F)\cup \C(G)$. Therefore, $F$ and $G$ are reducible mapping classes. Let $n,m$ denote the degrees of $F,G$, respectively, where $m$ is odd, and assume that $F,G$ are of infinite order. Since $\C(F)\cup \C(G)$ is a multicurve, comparing the multitwist components in $G^{-m}F^nG^m =F^{nk^m}$, it follows that $k^m = 1$. As $m$ is odd, $k =1$, which implies that $\langle F,G \rangle$ is abelian. Finally, when $F$ is periodic and $G$ is an infinite order reducible mapping class of degree $1$, by Theorem \ref{thm:main_thm1} (ii), we have that $\langle F,G\rangle$ is abelian.
\end{proof}

Now, we give several examples of infinite metacyclic subgroups of $\map(S_g)$ involving reducible generators. In the following example, we use the $n$-compatibility of cyclic actions to construct an infinite metacyclic subgroup of $\map(S_g)$ generated by a nontrivial periodic and a pseudo-periodic mapping class of infinite order.

\begin{exmp}
\label{exmp:zz2s3pp}
Let $\tilde{F},\tilde{G}\in \map(S_2)$ be periodic mapping classes with $$D_{\tilde{F}}=(3,0;(1,3),(1,3),(2,3)_1,(2,3)_1) \text{ and } D_{\tilde{G}}=(2,1;(1,2)_2,(1,2)_2),$$ respectively as in Figure \ref{fig:zz2s3pp}. From the theory developed in \cite{dhanwani}, there exist conjugates $F'$ and $G'$ of $\tilde{F}$ and $\tilde{G}$, respectively, that commute in $\map(S_g)$. We observe that the orbits corresponding to the cone points of $D_{\tilde{F}}$ (resp. $D_{\tilde{G}}$) with the same suffix are $1$-compatible with twist factor $1$ (resp. $1$-compatible). Hence, $F'$ and $G'$ extends to a pseudo-periodic $F$ and a periodic $G$ (represented by $\G$), respectively in $\map(S_3)$ such that $F^3=T_c$ and $D_G=(2,2,1;-)$. Since $\langle F',G'\rangle$ is abelian, from Theorem \ref{thm:main_thm1}, it follows that $\langle F,G \rangle \cong \Z \times \Z_2$.   
\begin{figure}[H]
\tiny
\begin{subfigure}{0.4\textwidth}
\centering
\labellist
\pinlabel $\G$ at 150 140
\pinlabel $\pi$ at 160 120
\pinlabel $c$ at 120 40
\endlabellist
\includegraphics[scale=0.6]{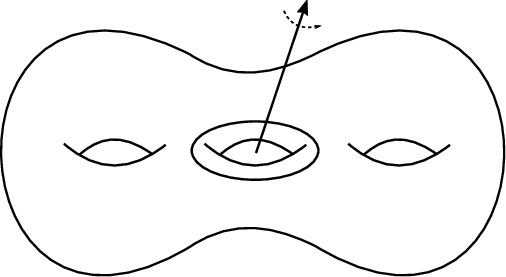}
\end{subfigure}
\begin{subfigure}{0.4\textwidth}
\centering
\labellist
\pinlabel $\G$ at 155 150
\pinlabel $\pi$ at 165 127
\pinlabel $c$ at 127 90
\pinlabel $(1,2)$ at 155 77
\pinlabel $(1,2)$ at 145 50
\pinlabel $(2,3)$ at 105 77
\pinlabel $(2,3)$ at 95 50
\pinlabel $(1,3)$ at 18 52
\pinlabel $(1,3)$ at 225 52
\endlabellist
\includegraphics[scale=0.6]{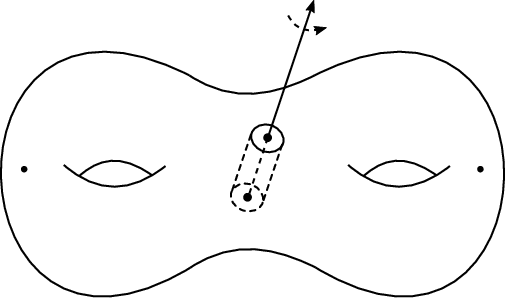}
\end{subfigure}
\caption{Realization of an infinite metacyclic subgroup $\langle F, G\rangle < \map(S_3)$ isomorphic to $\Z\times \Z_2$ generated by $F$ such that $F^3=T_c$ and a free involution $G$.}
\label{fig:zz2s3pp}
\end{figure}
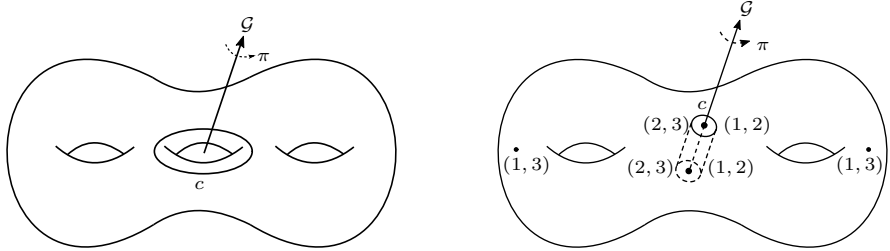
\end{exmp}

\noindent The following corollary is a direct generalization of Example \ref{exmp:zz2s3pp}.

\begin{cor}
For $g\geq 2$, let $F\in \map(S_g)$ be a nontrivial periodic mapping class with $$D_F=(n,g_0;(a,n),(b,n),(c_1,n_1),\dots,(c_{\ell},n_{\ell})).$$ For $1<m<n$ and $m\mid n$ such that $\gcd(m,n/m)=1$ , there is an infinite metacyclic subgroup of $\map(S_{g+1})$ isomorphic to $\Z \times \Z_m$ if the following conditions hold.
\begin{enumerate}[(i)]
\item $a+b \equiv 0 \pmod m$.
\item $a^{-1}+b^{-1} \equiv k \pmod {n/m}$, where $k\in \Z_{n/m}\setminus \{0\}$.
\end{enumerate}
\end{cor}

In the following example, we construct a non-abelian infinite metacyclic subgroup $\langle F,G\rangle$, where $F$ is a nontrivial reducible periodic mapping class.

\begin{exmp}
\label{exmp:zg-1zsgpp}
For $g\geq 1$, let $F_1,F_2\in \map(S_g)$ be two periodic mapping classes (see Figure \ref{fig:zg-1zsgpp}) with $$D_{F_1}=(2g+1,0;(1,2g+1),(g,2g+1)_1,(g,2g+1)_2) \text{ , and }$$ $$D_{F_2}=(2g+1,0;(2g,2g+1),(g+1,2g+1)_1,(g+1,2g+1)_2).$$ Since the orbits corresponding to the cone points with the same suffix are $1$-compatible, a periodic mapping class $F\in \map(S_{2g+1})$ can be constructed from $F_1,F_2$ with $$D_{F}=(2g+1,1;(1,2g+1),(2g,2g+1)).$$ Let $G'\in \map(S_{2g+1})$ be an involution represented by $\G'$ as shown in the figure with $D_{G'}=(2,g+1,1;-)$. From the theory developed in \cite{sanghi1}, we have $G'^{-1}FG'=F^{-1}$. Now, consider $G\in \map(S_{2g+1})$ such that $G=G'T_aT_b^{-1}$. Since $F(a)=a$ and $F(b)=b$, it follows that $G^{-1}FG=F^{-1}$. Hence, $\langle F,G \rangle \cong \Z_{2g+1} \rtimes_{-1} \Z $.
\begin{figure}[H]
\tiny
\labellist
\pinlabel $a$ at 225 0
\pinlabel $b$ at 223 120

\pinlabel $(1,2g+1)$ at -28 58
\pinlabel $(g,2g+1)$ at 162 100
\pinlabel $(g,2g+1)$ at 162 14

\pinlabel $(g+1,2g+1)$ at 300 17
\pinlabel $(g+1,2g+1)$ at 300 100
\pinlabel $(2g,2g+1)$ at 486 58

\pinlabel $\pi$ at 260 165
\pinlabel $\G'$ at 285 165
\endlabellist
\centering
\includegraphics[scale=0.63]{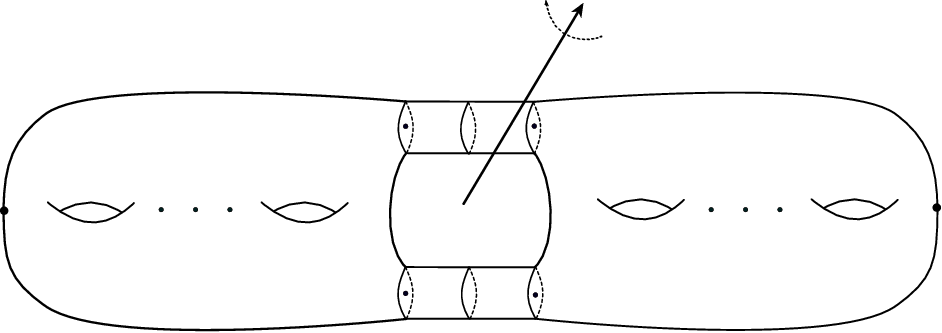}
\caption{Realization of an infinite metacylic subgroup $\langle F,G\rangle < \map(S_{{2g+1}})$ isomorphic to $\Z_{2g+1}\rtimes_{-1} \Z$ generated by a periodic mapping class $F$ of order $2g+1$ and $G$ such that $G^2=T_a^2T_b^{-2}$.}
\label{fig:zg-1zsgpp}
\end{figure}
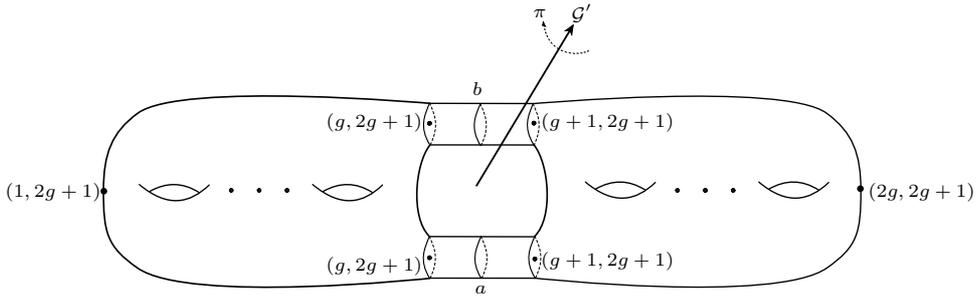
\end{exmp}

In \cite[Example 4.19]{sanghi1}, an infinite metacyclic subgroup $\langle F,G \rangle < \map(S_g)$ was constructed, where $F$ was an infinite-order pseudo-periodic and $G$ was a nontrivial periodic mapping class such that $\langle \G \rangle$ acted non-transitively on the path components of $S_g(\C(F))$. We now provide an example in which the action of $\langle \G \rangle$ on $S_g(\C(F))$ is transitive.

\begin{exmp}
\label{exmp:z-1z2frees5}
For $g\geq 2$, let $F_1,F_2\in \map(S_g)$ be periodic mapping classes (see Figure \ref{fig:z-1z2frees5}) with $$D_{F_1}=(2g+1,0;(g,2g+1)_1,(1,2g+1)_2,(g,2g+1)) \text{ and }$$ $$D_{F_2}=(2g+1,0;(2g,2g+1)_1,(g+1,2g+1)_2,(g+1,2g+1)).$$ Here, the orbits corresponding to cone points with the same suffix are $1$-compatible with twist factor $\pm 3$. Thus, there exist a pseudo-periodic $F\in \map(S_{2g+1})$ with $F_1,F_2$ as its canonical components such that $F^{2g+1}=T_a^3T_b^{-3}$, where $\C(F)=\{a,b\}$ is a bounding pair. Let $G\in \map(S_{2g+1})$ be represented by a free involution $\G$ as shown in the figure with $D_G=(2,g+1,1;-)$. From Theorem \ref{thm:main_thm1}, it follows that $GFG^{-1}=F^{-1}$, and hence, $\langle F,G \rangle \cong \mathbb{Z}\rtimes_{-1} \mathbb{Z}_2$.
\begin{figure}[H]
\tiny
\labellist
\pinlabel $a$ at 225 0
\pinlabel $b$ at 223 120

\pinlabel $(g,2g+1)$ at -30 58
\pinlabel $(g,2g+1)$ at 165 100
\pinlabel $(1,2g+1)$ at 163 14

\pinlabel $(g+1,2g+1)$ at 300 17
\pinlabel $(2g,2g+1)$ at 292 100
\pinlabel $(g+1,2g+1)$ at 492 58

\pinlabel $\pi$ at 260 165
\pinlabel $\G$ at 285 165
\endlabellist
\centering
\includegraphics[scale=0.63]{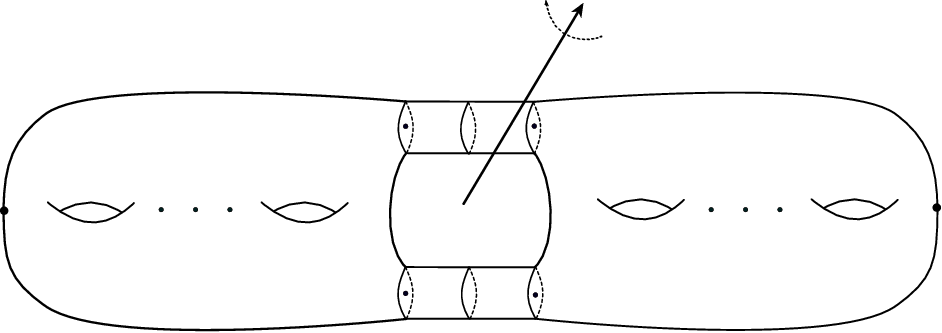}
\caption{Realization of an infinite metacyclic subgroup $\langle F,G \rangle <\map(S_{2g+1})$ isomorphic to $\mathbb{Z}\rtimes_{-1} \mathbb{Z}_2$ generated by an $F$ such that $F^{2g+1}=T_a^3T_b^{-3}$ and an involution $G$.}
\label{fig:z-1z2frees5}
\end{figure}
\end{exmp}

\noindent The constructions in Examples \ref{exmp:zg-1zsgpp} and \ref{exmp:z-1z2frees5} easily generalize to the following.

\begin{cor}
\label{cor:periodic_normal_1}
For $g\geq 2$, let $F \in \map(S_g)$ be a periodic mapping class with $$D_F = (n,g_0;(a,n),(b,n),(c_1,n_1), \dots, (c_{\ell},n_{\ell})),\text{ where }3\leq n\leq 4g.$$ Then the following statements hold.
\begin{enumerate}[(i)]
\item If $a\equiv b \pmod n$, then there is an infinite metacyclic subgroup of $\map(S_{2g+1})$ isomorphic to $\Z_n \rtimes_{-1} \Z$.
\item If $a\not \equiv b\pmod n$, then there is an infinite metacyclic subgroup of $\map(S_{2g+1})$ isomorphic to $\Z \rtimes_{-1} \Z_2$.
\end{enumerate}
\end{cor}

So far, we have only constructed infinite metacyclic subgroups with non-trivial periodic elements. In the next couple of examples, we construct infinite metacyclic subgroups that do not have any nontrivial periodic element.

\begin{exmp}
\label{exmp:zzsgpp}
For an odd integer $g>1$, let $G_1,G_2\in \map(S_g)$ be represented by a free involution $\G_1$ and a hyperelliptic involution $\G_2$ as in Figure \ref{fig:zzsgpp}. We observe that $\G_1$ and $\G_2$ commute. Consider $F_1,F_2,G\in \map(S_g)$ such that $F_1=G_2T_aT_b$, $F_2=G_2T_aT_b^{-1}T_eT_f^{-1}$, and $G=G_1 T_c$. Since $G^2=T_c^2$, $F_1^2=T_aT_bT_eT_f$, and $F_2^2=T_a^2T_b^{-2}T_e^2T_f^{-2}$, $F_1$, $F_2$, and $G$ are pseudo-periodic mapping classes. Now, it can be verified that $G^{-1}F_1G=F_1$ and $G^{-1}F_2G=F_2^{-1}$. Thus, we have $\langle F_1,G\rangle \cong \Z\times \Z$ and $\langle F_2, G \rangle \cong \Z\rtimes_{-1}\Z$.
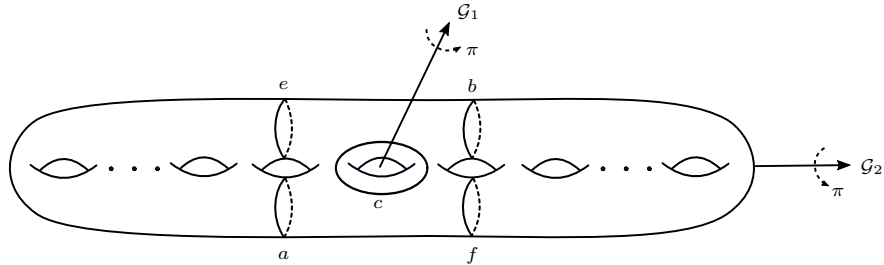
\begin{figure}[H]
\tiny
\labellist
\pinlabel $a$ at 160 -10
\pinlabel $f$ at 270 -10
\pinlabel $e$ at 160 90
\pinlabel $b$ at 270 90
\pinlabel $c$ at 215 20

\pinlabel $\pi$ at 270 110
\pinlabel $\G_1$ at 268 133
\pinlabel $\pi$ at 483 27
\pinlabel $\G_2$ at 503 42
\endlabellist
\centering
\includegraphics[scale=0.65]{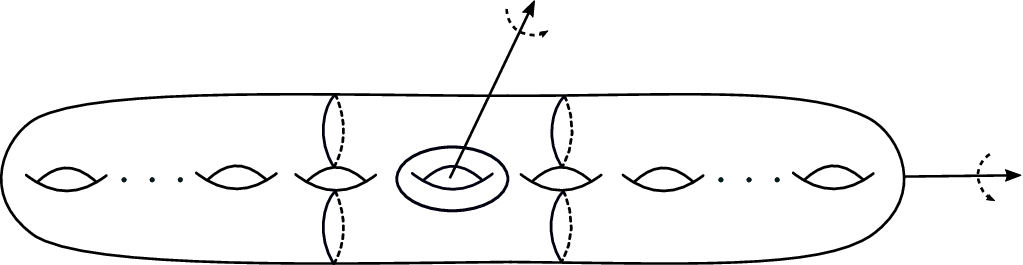}
\caption{Realization of an infinite metacyclic subgroup $\langle F,G\rangle < \map(S_g)$ isomorphic to $\Z\times \Z$ and $\Z\rtimes_{-1} \Z$ generated by two pseudo-periodic mapping classes $F$ and $G$.}
\label{fig:zzsgpp}
\end{figure}
\end{exmp}

\begin{exmp}
\label{exmp:z-1zsgpp}
For an even integer $g>2$, let $G_1,G_2\in \map(S_g)$ be represented by an involution $\G_1$ and a hyperelliptic involution $\G_2$ as in Figure \ref{fig:z-1zsgpp}. We observe that $G_1$ and $G_2$ commute. Consider $F,G\in \map(S_g)$ such that $F=G_2T_cT_d^{-1}$ and $G=G_1 T_aT_e$. Since $F^2=T_c^2T_d^{-2}$ and $G^2=T_aT_bT_eT_f$, $F$ and $G$ are pseudo-periodic mapping classes. It can be verified that $G^{-1}FG=F^{-1}$, and hence we have $\langle F,G\rangle \cong \Z\rtimes_{-1} \Z $. Considering $F' \in \map(S_g)$ such that $F' = G_2T_c T_d$, it can be seen that $\langle F',G \rangle \cong \Z \times \Z$.
\begin{figure}[H]
\tiny
\labellist
\pinlabel $a$ at 190 -10
\pinlabel $f$ at 240 -10
\pinlabel $e$ at 190 90
\pinlabel $b$ at 245 90
\pinlabel $c$ at 135 17
\pinlabel $d$ at 300 17

\pinlabel $\pi$ at 270 110
\pinlabel $\G_1$ at 268 133
\pinlabel $\pi$ at 483 27
\pinlabel $\G_2$ at 503 42
\endlabellist
\centering
\includegraphics[scale=0.65]{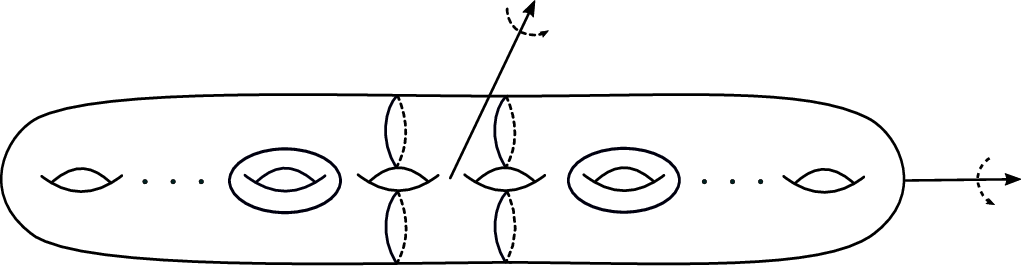}
\caption{Realization of an infinite metacyclic subgroup $\langle F,G\rangle < \map(S_g)$ isomorphic to $\Z\times \Z$ and $\Z\rtimes_{-1} \Z$ generated by two pseudo-periodic mapping classes $F$ and $G$.}
\label{fig:z-1zsgpp}
\end{figure}
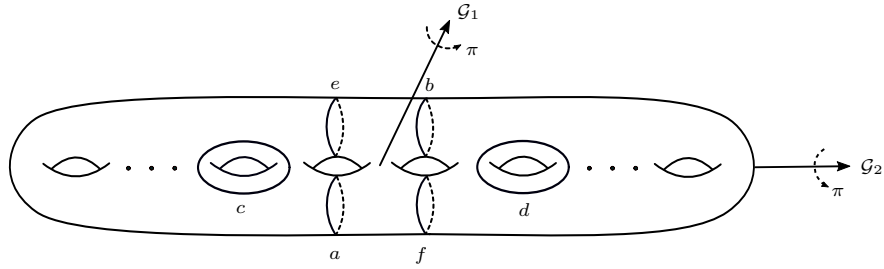
\end{exmp}

Taking inspiration from \cite[Example 4.20]{sanghi1}, where a non-abelian infinite metacyclic subgroup was constructed with a nontrivial periodic generator, we will now describe an example where both generators are pseudo-periodics of infinite order.

\begin{exmp}
\label{exmp:z-1zs13pp}
Let $F,G'\in \map(S_{13})$ such that $F^3=T_{c_1}T_{c_2}^{-1}T_{c_3}T_{c_4}^{-1}$ and $G'$ represented by $\G'$ with $D_{\G'} = (4,4,1;-)$ (see Figure \ref{fig:z-1zs13pp}). In \cite[Example 4.20]{sanghi1}, it was shown that $G'FG'^{-1} = F^{-1}$, and therefore $\langle F, G' \rangle \cong \Z \rtimes_{-1} \Z_4 < \map(S_{13})$. Now, we consider $G\in \map(S_{13})$ such that $G=G'T_{c_1}$. Since $G^4=T_{c_1}T_{c_2}T_{c_3}T_{c_4}$, the $G$ is pseudo-periodic of degree $4$. As $F(c_1)=c_1$ and $$G^{-1}FG=T_{c_1}^{-1}G'^{-1}FG'T_{c_1}=T_{c_1}^{-1}F^{-1}T_{c_1}=F^{-1},$$ we have $\langle F,G\rangle \cong \Z\rtimes_{-1}\Z$.  
\begin{figure}[H]
\centering
\labellist
\tiny
\pinlabel $\G'$ at 275 595
\pinlabel $\frac{\pi}{2}$ at 300 555
\pinlabel $(1,3)$ at 295 485
\pinlabel $(1,3)$ at 215 460
\pinlabel $(1,3)$ at 160 465
\pinlabel $(2,3)$ at 350 485
\pinlabel $(2,3)$ at 410 455
\pinlabel $(2,3)$ at 470 467
\pinlabel $(1,3)$ at 270 315
\pinlabel $(1,3)$ at 330 340
\pinlabel $(1,3)$ at 310 300
\pinlabel $(2,3)$ at 70 300
\pinlabel $(2,3)$ at 80 350
\pinlabel $(2,3)$ at 115 315
\pinlabel $c_1$ at 500 350
\pinlabel $c_2$ at 185 340
\pinlabel $c_3$ at 100 185
\pinlabel $c_4$ at 345 185
\pinlabel $(1,3)$ at 255 208
\pinlabel $(1,3)$ at 215 225
\pinlabel $(1,3)$ at 190 180
\pinlabel $(2,3)$ at 420 205
\pinlabel $(2,3)$ at 455 180
\pinlabel $(2,3)$ at 470 233
\pinlabel $(1,3)$ at 290 45
\pinlabel $(1,3)$ at 300 67
\pinlabel $(1,3)$ at 365 80
\pinlabel $(2,3)$ at 68 65
\pinlabel $(2,3)$ at 120 80
\pinlabel $(2,3)$ at 135 45
\endlabellist		
\includegraphics[scale=0.45]{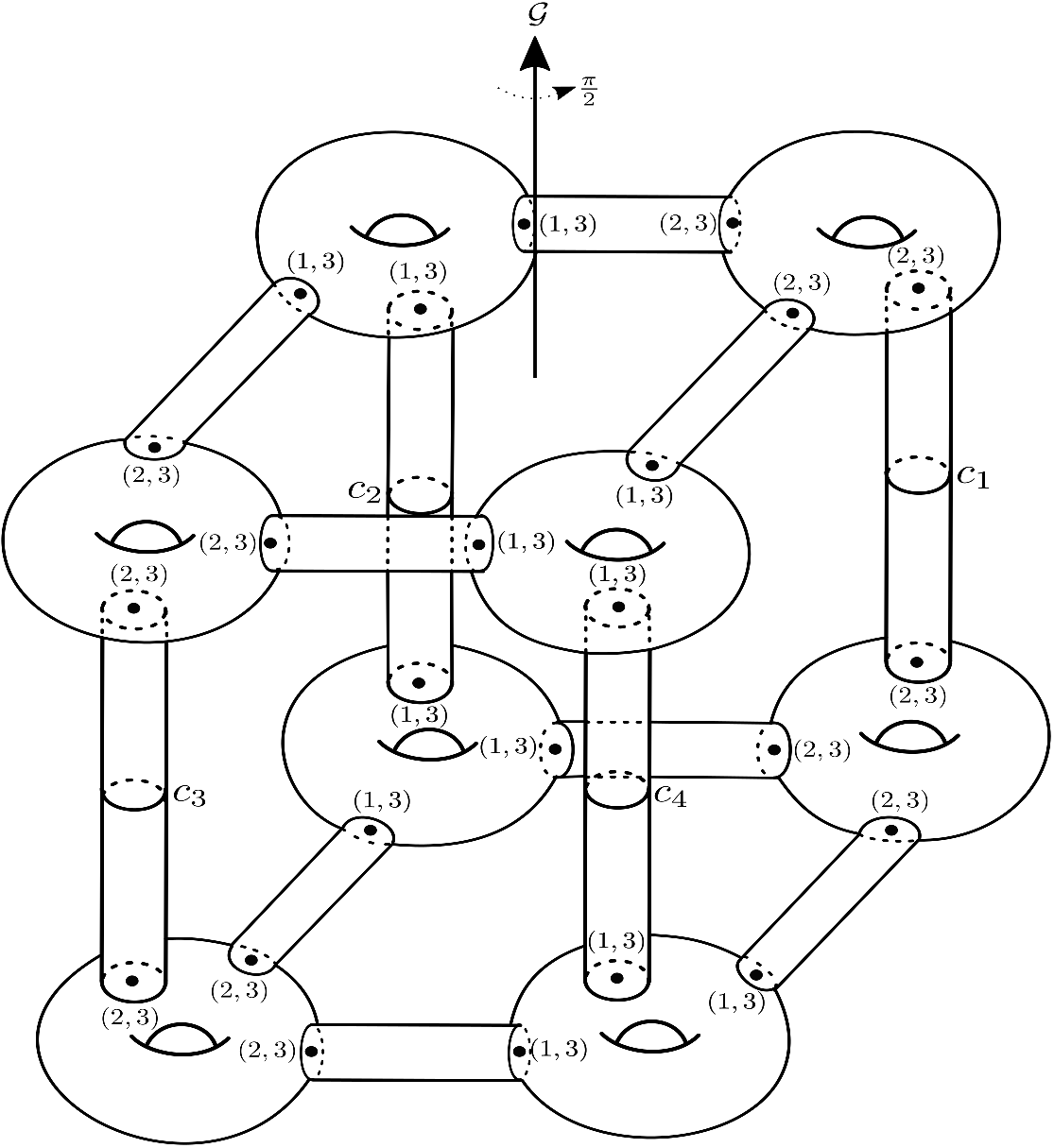}
\caption{Realization of an infinite metacyclic subgroup of $\map(S_{13})$ isomorphic to $\Z \rtimes_{-1} \Z$.}
\label{fig:z-1zs13pp}
\end{figure}
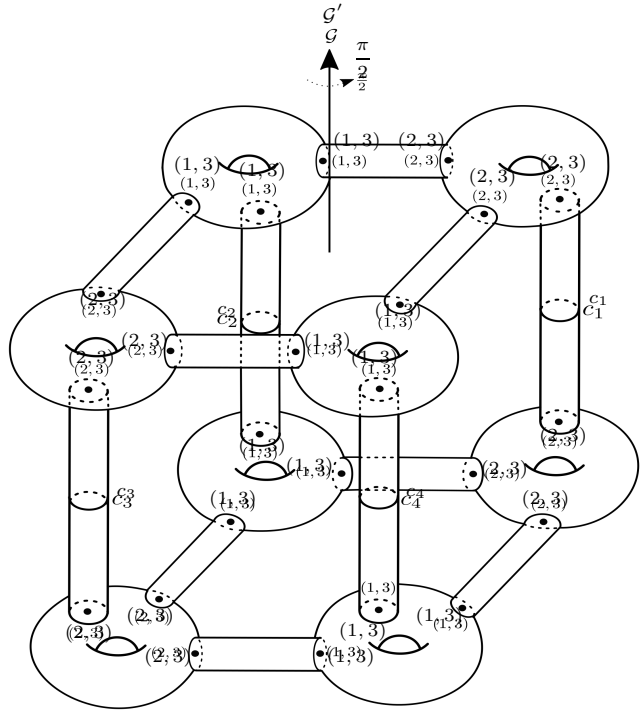
\end{exmp}

In the preceding examples, we saw infinite metacyclic subgroups with pseudo-periodic generators. In the following examples, we construct infinite metacyclic subgroups with an infinite order reducible generator with canonical components that are nontrivial periodic and pseudo-Anosov.

\begin{exmp}
\label{exmp:zz2sg}
For $g\geq 3$, consider the collection of curves as shown in Figure \ref{fig:zz2sg} and the mapping class $$F=T_{b_1}T_{b_2}T_{a_1}^{-1}T_{a_2}^{-1}T_{b_3} \prod_{i=3}^{g}T_{a_i}T_{b_{i+1}}.$$ Since $F(b_3)=b_3$, $F$ is a reducible mapping class of infinite order with pseudo-Anosov canonical component $T_{b_1}T_{b_2}T_{a_1}^{-1}T_{a_2}^{-1}$ and periodic canonical component $\textstyle \prod_{i=3}^{g}T_{a_i}T_{b_{i+1}}$. Let $G$ be the hyperelliptic involution as shown in the figure. Since $G^{-1}FG=F$, we have $\langle F,G\rangle\cong \Z\times \Z_2$.
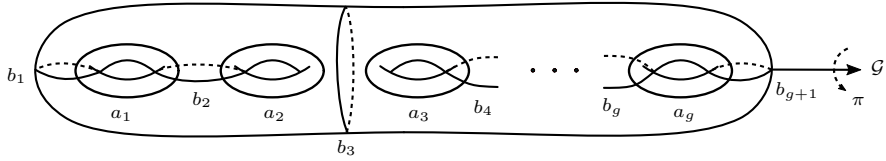
\begin{figure}[H]
\centering
\labellist
\tiny
\pinlabel $b_1$ at -10, 40
\pinlabel $b_2$ at 105, 25
\pinlabel $b_3$ at 195, -5
\pinlabel $b_4$ at 280, 20
\pinlabel $b_g$ at 360, 20
\pinlabel $b_{g+1}$ at 475, 30

\pinlabel $a_1$ at 55, 15
\pinlabel $a_2$ at 150, 15
\pinlabel $a_3$ at 240, 15
\pinlabel $a_g$ at 405, 15

\pinlabel $\G$ at 525, 45
\pinlabel $\pi$ at 513, 24
\endlabellist
\includegraphics[scale=0.6]{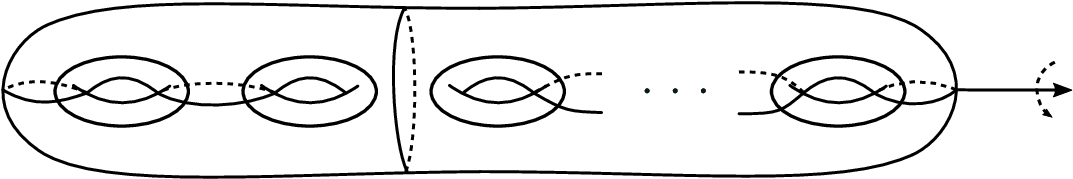}
\caption{Realization of an infinite metacyclic subgroup $\langle F,G\rangle < \map(S_g)$ isomorphic to $\mathbb{Z}\times \mathbb{Z}_2$ generated by a hyperelliptic involution $G$ and a reducible mapping class $F$ of infinite order.}
\label{fig:zz2sg}
\end{figure}
\end{exmp}

\noindent The construction in Example \ref{exmp:zz2sg} generalizes to the following assertion.

\begin{cor}
For $g\geq 2$, there is an infinite metacyclic subgroup $\langle F,G \rangle<\map(S_g)$ isomorphic to $\Z\times \Z_2$ generated by a hyperelliptic involution $G$ and a reducible mapping class of infinite order containing at least one pseudo-Anosov and one nontrivial periodic canonical component. 
\end{cor}

\begin{exmp}
\label{exmp:z-1z2s6}
Consider the collection of curves in $S_6$ as shown in Figure \ref{fig:z-1z2s6} and the mapping class $$F=(T_{b_1}T_{a_1}T_{b_2})(T_{c_1}T_{c_2}^{-1})(T_{b_3}T_{d_1}T_{a_4}T_{a_3}^{-1}T_{d_2}^{-1}T_{b_5}^{-1})(T_{c_3}T_{c_4}^{-1})(T_{b_6}^{-1}T_{a_6}^{-1}T_{b_7}^{-1}).$$ Since $F(\{c_1,c_2,c_3,c_4\})=\{c_1,c_2,c_3,c_4\}$, $F$ is a reducible mapping class of infinite order with two nontrivial periodic canonical components and one pseudo-Anosov canonical component. Let $G\in \map(S_6)$ be an involution as shown in the figure. Since $G^{-1}FG=F^{-1}$, $\langle F,G\rangle\cong \Z\rtimes_{-1} \Z_2$.
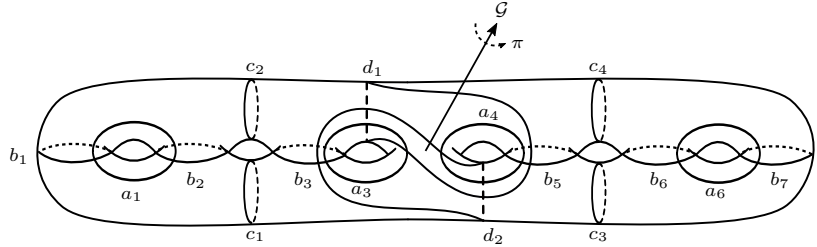
\begin{figure}[H]
\centering
\labellist
\tiny
\pinlabel $b_1$ at -10, 45
\pinlabel $b_2$ at 100, 30
\pinlabel $b_{3}$ at 167, 30
\pinlabel $b_5$ at 322, 30
\pinlabel $b_6$ at 388, 30
\pinlabel $b_7$ at 463, 30

\pinlabel $a_1$ at 60, 20
\pinlabel $a_3$ at 203, 23
\pinlabel $a_4$ at 282, 71
\pinlabel $a_6$ at 423, 22

\pinlabel $c_1$ at 137, -5
\pinlabel $c_2$ at 137, 100
\pinlabel $c_3$ at 350, -5
\pinlabel $c_4$ at 350, 100

\pinlabel $d_1$ at 210, 100
\pinlabel $d_2$ at 285, -5

\pinlabel $\pi$ at 300, 115
\pinlabel $\G$ at 290, 137
\endlabellist
\includegraphics[scale=0.6]{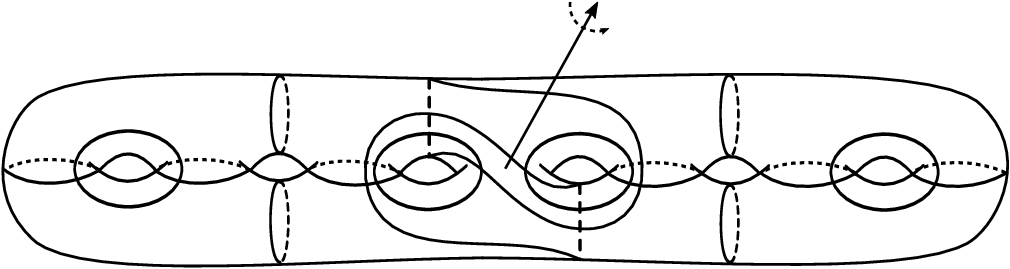}
\caption{Realization of an infinite metacyclic subgroup $\langle F,G\rangle < \map(S_6)$ isomorphic to $\mathbb{Z}\rtimes_{-1} \mathbb{Z}_2$ generated by an involution $G$ and a reducible mapping class $F$ of infinite order.}
\label{fig:z-1z2s6}
\end{figure}
\end{exmp}

\noindent A direct generalization of Example \ref{exmp:z-1z2s6} is the following result.

\begin{cor}
For an even integer $g\geq 4$, there is an infinite metacyclic subgroup $\langle F,G \rangle<\map(S_g)$ isomorphic to $\Z\rtimes_{-1} \Z_2$ generated by an involution $G$ with $D_G=(2,g/2;(1,2),(1,2))$ and a reducible mapping class of infinite order containing at least one pseudo-Anosov and one nontrivial periodic canonical components.
\end{cor}

\section{Applications}
\label{sec:application}
In this section, we derive some applications of the theory developed in this paper.
\subsection{Infinite metacyclic subgroups of the level $m$ subgroup of $\map(S_g)$}
\label{subsec:level_m}
The action of $\map(S_g)$ on $H_1(S_g,\mathbb{Z})$ affords a surjective representation \cite[Chapter 6]{primer} $\Psi:\map(S_g)\longrightarrow \mathrm{Sp}(2g,\mathbb{Z})$. The subgroup $\ker \Psi$ is known as the \textit{Torelli group} and is denoted by $\mathcal{I}(S_g)$. Further, for an integer $m\geq 2$, \textit{the level $m$ congruence subgroup} is the kernel of the composition
\begin{center}
$\map(S_g)\longrightarrow \mathrm{Sp}(2g,\mathbb{Z})\longrightarrow \mathrm{Sp}(2g,\Z_m)$,
\end{center}
denoted by $\map(S_g)[m]$. By definition $\mathcal{I}(S_g)\subset \map(S_g)[m]$ for every $m$. For $m\geq 3$, it is known \cite[Chapter 6]{primer} that $\map(S_g)[m]$ is torsion free and that an infinite order reducible in $\map(S_g)[m]$ has degree $1$ \cite[Corollary 1.8]{ivanov}. The only torsion elements of $\map(S_g)[2]$ are the hyperelliptic involutions. The following result follows immediately from Theorem \ref{thm:main_thm2}, \ref{thm:main_thm1}, and Corollary \ref{cor:main}.

\begin{prop}
\label{prop:sub_level_m}
For $g\geq 2$ and $m\geq 3$, let $F,G \in\map(S_g)[m]$ be two nontrivial mapping classes. Then $\langle F,G\rangle$ is metacyclic with $\langle F\rangle\lhd\langle F,G\rangle$ if and only if the following hold.
\begin{enumerate}[(i)]
\item $F$ and $G$ are infinite order reducible mapping classes that commute.
\item $\C(F)\cup \C(G)$ is a multicurve.
\item The nontrivial canonical components of $F$ and $G$ are pseudo-Anosov mapping classes.
\item The nontrivial canonical components of $F$ and $G$ with the same support generate a cyclic group.
\end{enumerate}
\end{prop}

\noindent In the following examples, we construct infinite metacyclic subgroups of $\map(S_g)[2]$ with a pseudo-Anosov generator. Since hyperelliptic involution of $\map(S_2)$ lies in the center, we will assume $g\geq 3$.

\begin{exmp}
\label{exmp:hyperelliptic}
Consider the multicurves $A=\{a_1,a_2,\dots,a_g\}$ and $B=\{b_1,b_2,\dots,b_{g+1}\}$ as shown in the Figure \ref{fig:hyperelliptic}. Since curves of $A \cup B$ fills $S_g$, by Theorem \ref{thm:penner}, the mapping class $$F=\prod_{i=1}^gT_{a_i}^2\prod_{i=i}^{g+1}T_{b_i}^{-2}$$ is pseudo-Anosov. Let $G$ be the hyperelliptic involution as shown in Figure \ref{fig:hyperelliptic}. Since $G(c)=c$ for every $c\in A\cup B$, we have $G^{-1}FG=F$. As $F,G\in \map(S_g)[2]$, $\langle F,G \rangle < \map(S_g)[2]$ isomorphic to $\Z\times \Z_2$.
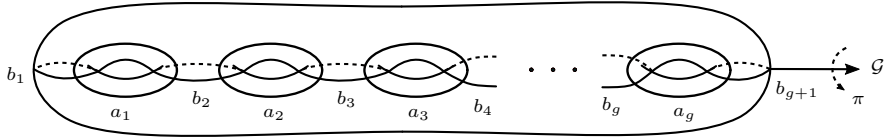
\begin{figure}[H]
\labellist
\tiny

\pinlabel $b_1$ at -10, 40
\pinlabel $b_2$ at 105, 25
\pinlabel $b_3$ at 195, 25
\pinlabel $b_4$ at 280, 20
\pinlabel $b_g$ at 360, 20
\pinlabel $b_{g+1}$ at 475, 30

\pinlabel $a_1$ at 55, 15
\pinlabel $a_2$ at 150, 15
\pinlabel $a_3$ at 240, 15
\pinlabel $a_g$ at 405, 15

\pinlabel $\G$ at 525, 45
\pinlabel $\pi$ at 513, 24

\endlabellist
\centering
\includegraphics[scale=0.6]{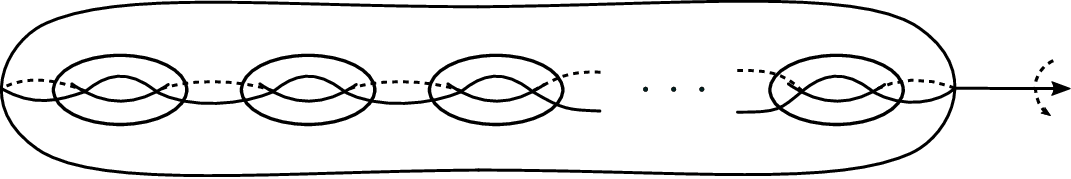}
\caption{Realization of an infinite metacyclic subgroup $\langle F,G\rangle< \map(S_g)[2]$ isomorphic to $\mathbb{Z}\times \mathbb{Z}_2$ generated by a hyperelliptic involution $G$ and a pseudo-Anosov $F$.}
\label{fig:hyperelliptic}
\end{figure}
\end{exmp}

\noindent The following example draws inspiration from a family of Penner-type pseudo-Anosov mapping classes described in~\cite{chris}.

\begin{exmp}
\label{exmp:pa_tg_level2}
For $g \geq 3$, we construct a non-abelian metacyclic subgroup of $\map(S_g)[2]$ generated by a pseudo-Anosov element in $\mathcal{I}(S_g)$. First we describe a filling collection of curves $C$ in $S_g$ which is a disjoint union of two multicurves $A$ and $B$. Consider the surfaces $S$, $S'$, and $S''$ with curves and arcs as shown in the Figure \ref{fig:pa_tg_1}. We construct a closed surface by combining multiple copies of $S$, $S'$, and $S''$ as follows. For $X_i\in \{S,S',S''\}$, we write $X_1+X_2\dots +X_n$ for the surface obtained by gluing $X_i$ end to end and capping the remaining boundary components after gluing. For $m\geq 1$, we write $mS$ for $S+S+\dots+S$. For $m\geq 1$, we can write $S_g=mS$ if $g=3m$, $S_g=S+S'+mS$ if $g=3m+4$, $S_g=S+S'+mS+S'+S$ if $g=3m+8$, $S_5=S''$, and $S_8=S''+S'$. The multicurves $A$ and $B$ are drawn with red and blue color, respectively. 
\begin{figure}[t]
\centering
\includegraphics[scale=0.7]{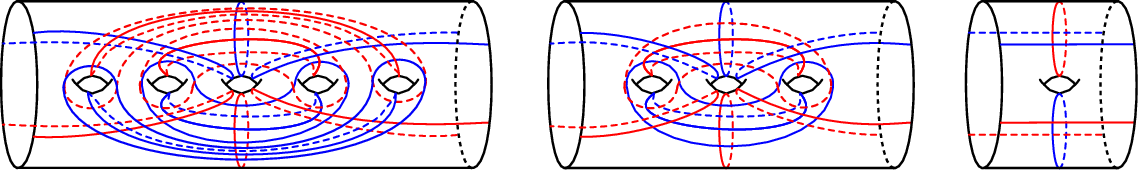}
\caption{The surface $S''$ on the left, $S$ in center, and $S'$ on the right used to construct a filling system of curves in $S_g$.}
\label{fig:pa_tg_1}
\end{figure}
\begin{figure}[t]
\centering
\labellist
\tiny

\pinlabel $\G$ at 470, 35
\pinlabel $\pi$ at 462, 47

\endlabellist
\includegraphics[scale=0.85]{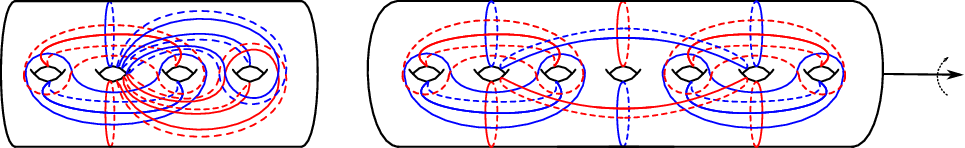}
\caption{A filling system of curves in $S_4$ and $S_7=S+S'+S$.}
\label{fig:pa_tg_2}
\end{figure}
\noindent Let $F$ be a product of positive (left-handed) Dehn twists about the curves in $A$ and negative Dehn twists about the curves in $B$, where each Dehn twist is taken exactly once. We observe that for each curve $a\in A$, there exist a unique curve $b\in B$ such that $\{a,b\}$ bounds a subsurface of $S_g$ and vice-versa. Since $A \cup B$ fills $S_g$, by Theorem \ref{thm:penner}, $F$ is a pseudo-Anosov. It can be seen that $F\in \mathcal{I}(S_g)$. Let $G$ be the hyperelliptic involution shown in Figure \ref{fig:pa_tg_2}. Since $G$ exchanges multicurves $A$ and $B$, we have $G^{-1}FG=F^{-1}$. Hence, $\langle F,G\rangle <\map(S_g)[2]$ is isomorphic to $\Z\rtimes_{-1} \Z_2$. 
\end{exmp}

\subsection{Bounds on the order of a periodic generator of an infinite metacyclic group}
\label{subsec:bounds}
In this subsection, we derive bounds on the order of a nontrivial periodic generator of an infinite metacyclic subgroup of $\map(S_g)$ that are realized. We will require the following remark.

\begin{rem}
\label{rem:2g+2}
For an even integer $g$, let $F\in \map(S_g)$ be a nontrivial reducible periodic mapping class of order $2g+2$. From the theory developed in \cite{rajeevsarathy1}, it follows that $F$ arises as a $1$-compatibility between fixed points of $F'$ and $F'^{-1}$, where $F'\in \map(S_{g/2})$ is a periodic mapping class of order $4(g/2)+2=2g+2$. Hence, $F$ has a unique maximal reduction system containing a single separating curve.  
\end{rem}

In the following proposition, we obtain bounds on the order of a nontrivial periodic generator that are realized.

\begin{prop}
\label{prop:bounds}
For $g\geq 2$, let $F,G \in \map(S_g)$ be two nontrivial mapping classes such that $\langle F,G\rangle$ is an infinite metacyclic subgroup with $\langle F\rangle\lhd\langle F,G\rangle$.
\begin{enumerate}[(i)]
\item Let $F$ be a pseudo-Anosov mapping class and $G$ be a periodic mapping class.
\begin{enumerate}
\item If $\langle F,G\rangle$ is non-abelian, then $2\leq |G|\leq 4g$.
\item If $\langle F,G\rangle$ is abelian, then $2\leq |G|\leq 2g$.
\end{enumerate}
\item Let $F$ be an reducible mapping class of infinite order and $G$ be a periodic mapping class.
\begin{enumerate}
\item If $\langle F,G\rangle$ is abelian, then $2\leq |G|\leq 2g+2$.
\item If $\langle F,G\rangle$ is non-abelian, then $2\leq |G|\leq 2g$.
\end{enumerate}
\item If $F$ is periodic and $\langle F,G\rangle$ is non-abelian, then $3\leq |F|\leq 2g+2$.
\end{enumerate}
Moreover, all of the above bounds are realized.
\end{prop}

\begin{proof}
\begin{enumerate}[(i)]
\item Suppose that $\langle F,G\rangle$ is non-abelian. Example \ref{cor:polygon} shows that an order $4g$ periodic mapping class can form a non-abelian metacyclic subgroup with $F$. Since it is known that $|G| \leq 4g+2$~\cite{harvey} and there is no periodic mapping class of order $4g+1$ (Lemma \ref{lem:4g+1}), it suffices to show that $|G| \neq 4g+2$. If $|G| = 4g+2$, then by Lemma \ref{lem:reducible_bound}, $G^2$ is irreducible. Furthermore, from Theorem \ref{thm:gilman}, it follows that $\Orb_{G^2} \approx S_{0,3}$. Since $G^2$ commutes with $F$, by Lemma \ref{lem:ind_auto}, $F$ induces an infinite order mapping class in the finite group $\map(S_{0,3})$, which is impossible. Thus, it follows that $|G| \leq 4g$.

Next, we consider the case when $\langle F,G \rangle$ is abelian. The preceding argument shows that $G$ is a reducible mapping class. By Lemma \ref{lem:reducible_bound}, it follows that $|G|\leq 2g+2$. Moreover, Example \ref{cor:polygon} shows that an order $2g$ periodic mapping class can form an infinite abelian metacyclic subgroup with $F$. Since $G$ is reducible, it suffices to show that $|G|\neq 2g+2$. Let $G$ be a reducible periodic mapping class of order $2g+2$ that commutes with $F$. By Remark~\ref{rem:2g+2}, $G$ has a unique maximal reduction system containing a single separating curve, say, $c$. Since $GF=FG$, we have $GF(c)=F(c)$, and so $F(c)=c$, which is not possible as $F$ is irreducible. Thus, it follows that $|G| \leq 2g$. Examples \ref{exmp:hyperelliptic} - \ref{exmp:pa_tg_level2} show that the lower bounds are realized.

\item Let $\langle F,G \rangle$ be abelian. From Corollary \ref{cor:main}, it follows that $G$ is reducible, and from Lemma \ref{lem:reducible_bound}, we have $|G|\leq 2g+2$. As before, a periodic mapping class of order $2g+2$ has a unique maximal reduction system $\{c\}$. Taking $F=T_c$, we have $G$ commutes with $F$. This shows that the upper bound $2g+2$ is realized when $\langle F,G \rangle$ is abelian.

Let $\langle F,G \rangle$ be non-abelian. If $|G|=2g+2$, then from Theorem \ref{thm:main_thm1}, it follows that $G(\C(F))=\C(F)$. Since $G$ has a unique maximal reduction system $\{c\}$, it follows that $\C(F)=\{c\}$ (as $\C(F)\neq \emptyset$). Therefore, the multitwist component of $F^n$ is $T_c^q$, for some $n\in \N$ and $q\in \Z\setminus \{0\}$. By comparing multitwist components in $G^{-1}F^nG=F^{-n}$, we have $G^{-1}T_c^qG=T_c^{-q}$. This is impossible since $G$ commutes with $T_c$. Hence, $|G|\leq 2g$ and Example \ref{cor:polygon_pp} shows that this upper bound is realized. Examples \ref{exmp:z-1z2frees5} and \ref{cor:polygon_pp} show that the lower bounds are realized.

\item Since $\langle F,G \rangle$ is non-abelian and $\langle F\rangle \lhd \langle F,G\rangle$, we have $k\in \Z_n^{\times}\setminus \{1\}$, where $n=|F|$. Hence, $n\geq 3$, and by Theorem \ref{thm:main_thm2} and Corollary \ref{cor:main}, $F$ is reducible. From Lemma \ref{lem:reducible_bound}, we have $n\leq 2g+2$. Thus, the assertion follows, and by Corollaries \ref{cor:periodic_normal_1} and \ref{cor:periodic_normal_2}, it follows that the bounds are realized.    
\end{enumerate}
\end{proof}

We will now provide examples demonstrating that the upper bound on the order of the periodic generator $G$ of the group $\langle F,G \rangle$ obtained in Proposition \ref{prop:bounds} is realized. 

\begin{exmp}
\label{cor:polygon}
For $g\geq 2$, let $G\in \map(S_g)$ be a periodic mapping class of order $4g$ realized as $2\pi/4g$-rotation of a $4g$-gon with side-pairing $a_1a_2\cdots a_{2g}a_1^{-1}a_2^{-1}\cdots a_{2g}^{-1}$ as shown in Figure \ref{fig:polygon} (for $g=2$).
\begin{figure}[h]
\labellist
\small

\pinlabel $a_1$ at 140, 5
\pinlabel $a_2$ at 255, 40
\pinlabel $a_3$ at 315, 130
\pinlabel $a_4$ at 295, 245

\pinlabel $a_1$ at 130, 310
\pinlabel $a_2$ at 35, 250
\pinlabel $a_3$ at 0, 120
\pinlabel $a_4$ at 65, 35

\pinlabel $c_1$ at 200, 50
\pinlabel $c_2$ at 265, 120
\pinlabel $c_3$ at 265, 200
\pinlabel $c_4$ at 200, 260

\pinlabel $c_1$ at 110, 265
\pinlabel $c_2$ at 60, 200
\pinlabel $c_3$ at 60, 110
\pinlabel $c_4$ at 120, 50

\pinlabel $\G$ at 160, 180
\pinlabel $\frac{2\pi}{8}$ at 185, 155

\endlabellist
\centering
\includegraphics[scale=0.45]{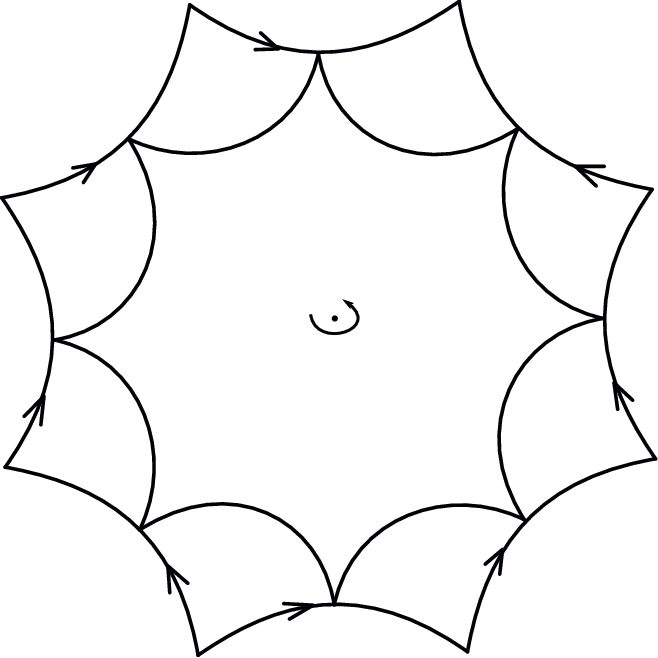}
\caption{Realization of an infinite metacyclic subgroup $\langle F,G\rangle < \map(S_2)$ isomorphic to $\mathbb{Z}\rtimes_{-1} \mathbb{Z}_8$ generated by an irreducible periodic mapping class $G$ of order $8$ and a pseudo-Anosov mapping class $F$.}
\label{fig:polygon}
\end{figure}
\noindent For $1\leq i\leq 2g-1$, let $c_i=a_ia_{i+1}$, $c_{2g}=a_{2g}a_1^{-1}$, $A=\{c_1,c_3,\dots,c_{2g-1}\}$, and $B=\{c_2,c_4,\dots,c_{2g}\}$. We note that $c_i$ is homotopic to the concatenation of $a_i$ and $a_{i+1}$. We observe that $A$ and $B$ are multicurves such that the curves in $A\cup B$ fill $S_g$. By Theorem \ref{thm:penner}, the mapping class $$F=\prod_{k=1}^gT_{c_{2k-1}}\prod_{k=1}^gT_{c_{2k}}^{-1}$$ is pseudo-Anosov. Since $G(c_i)=c_{i+1}$, where $1\leq i\leq 2g-1$ and $G(c_{2g})=c_1$, we have $G^{-1}FG=F^{-1}$. Hence, $\langle F,G\rangle \cong \mathbb{Z}\rtimes_{-1} \mathbb{Z}_{4g}$. Furthermore, since $G^{-2}FG^2=F$, we have $\langle F,G^2\rangle\cong \mathbb{Z}\times \mathbb{Z}_{2g}$.
\end{exmp}

\begin{exmp}
\label{cor:polygon_pp}
For even integer $g\geq 2$, let $G\in \map(S_g)$ be a periodic of order $2g$ realized as the square of $2\pi/4g$-rotation of a $4g$-gon with side-pairing $a_1a_2\cdots a_{2g}a_1^{-1}a_2^{-1}\cdots a_{2g}^{-1}$ as shown in Figure \ref{fig:polygon} (for $g=2$). Consider the  multicurve $\C = \{c_{2i-1}:=a_{2i-1}a_{2i} ~|~ 1 \leq i \leq g\}$, where $c_{2i-1}$ is homotopic to the concatenation of $a_{2i-1}$ and $a_{2i}$. Define the multitwist $$F=\prod_{i=1}^{g/2}T_{c_{4i-3}}T_{c_{4i-1}}^{-1}.$$ For $1\leq i\leq g-1$, as $G(c_{2i-1})=c_{2i+1}$ and $G(c_{2g-1}) = c_{1}$, we have $G^{-1}FG=F^{-1}$ and $G^{-g}FG^g=F$. Therefore, it follows that $\langle F,G\rangle \cong \mathbb{Z}\rtimes_{-1} \mathbb{Z}_{2g}$ and $\langle F,G^g\rangle \cong \mathbb{Z}\times\mathbb{Z}_{2}$.
\end{exmp}

The following example shows that the upper bound on the order of the periodic generator $F$ of a non-abelian metacyclic subgroup $\langle F, G \rangle$ obtained in Proposition \ref{prop:bounds} is realized when $G$ is reducible of infinite order.

\begin{exmp}
\label{exmp:z6-1zs2}
Let $F_1,F_2\in \map(S_1)$ be periodic with $$D_{F_1}=(6,0;(1,2),(1,3),(1,6)_1) \text{ and } D_{F_2}=(6,0;(1,2),(2,3),(5,6)_1).$$ Since the orbits corresponding to cone points with the same suffix are $1$-compatible, a periodic mapping class $F\in \map(S_2)$ can be constructed from $F_1,F_2$ with $$D_F=(6,0;(1,2),(1,2),(1,3),(2,3)).$$ Let $G'\in \map(S_2)$ be an involution represented by $\G'$ as shown in Figure \ref{fig:z6-1zs2} with $D_{G'}=(2,1;(1,2),(1,2))$. From the theory developed in \cite{sanghi1}, it follows that $G'^{-1}FG'=F^{-1}$. Now, consider $G\in \map(S_2)$ such that $G=G'T_c$. Since $F(c)=c$, we have $G^{-1}FG=T_c^{-1}G'^{-1}FG'T_c=F^{-1}$, and hence, $\langle F,G\rangle \cong \Z_6\rtimes_{-1}\Z$.
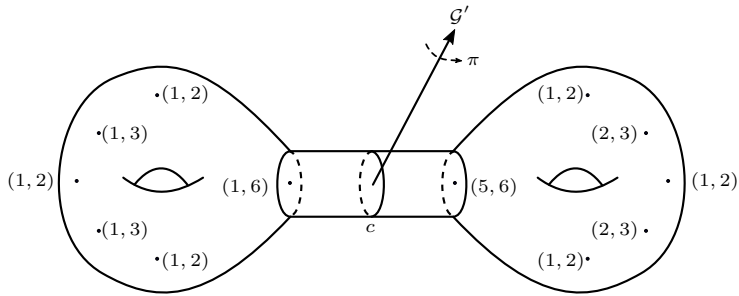
\begin{figure}[H]
\labellist
\tiny
\pinlabel $(5,6)$ at 325, 80
\pinlabel $(2,3)$ at 415, 120
\pinlabel $(2,3)$ at 415, 48
\pinlabel $(1,2)$ at 375, 150
\pinlabel $(1,2)$ at 375, 25
\pinlabel $(1,2)$ at 490, 85
	
\pinlabel $(1,6)$ at 140, 80
\pinlabel $(1,3)$ at 50, 120
\pinlabel $(1,3)$ at 50, 48
\pinlabel $(1,2)$ at 95, 150
\pinlabel $(1,2)$ at 95, 25
\pinlabel $(1,2)$ at -20, 85
	
\pinlabel $\pi$ at 310, 175
\pinlabel $\G'$ at 300, 210
\pinlabel $c$ at 233, 50
\endlabellist
\centering
\includegraphics[scale=0.5]{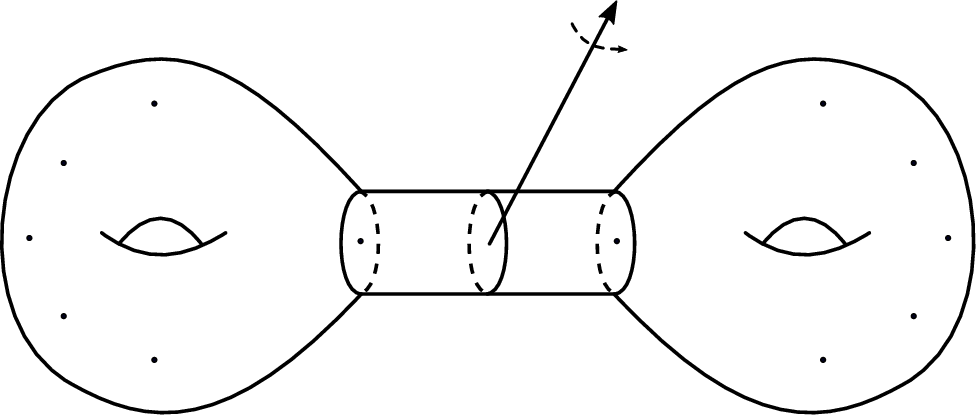}
\caption{An infinite metacyclic subgroup $\langle F,G\rangle<\map(S_{2})$ isomorphic to $\Z_6 \rtimes_{-1} \Z$ generated by an order $6$ mapping class $F$ and a pseudo-periodic mapping class $G$.}
\label{fig:z6-1zs2}
\end{figure}
\end{exmp}

\noindent Example \ref{exmp:z6-1zs2} generalizes to the following corollary.

\begin{cor}
\label{cor:periodic_normal_2}
For an even integer $g \geq 2$, there is an infinite metacyclic subgroup $\langle F,G \rangle <\map(S_g)$ isomorphic to $\Z_{2g+2}\rtimes_{-1} \Z$, where $F$ is a nontrivial periodic mapping class and $G$ is a pseudo-periodic mapping class of infinite order. 
\end{cor}

\subsection{Types of elements in an infinite metacyclic group}
\label{subsec:elements_type}
Let $\langle F,G\rangle < \map(S_g)$ be an infinite metacyclic subgroup with $\langle F\rangle\lhd\langle F,G\rangle$. In this subsection, we determine the Nielsen-Thurston type of the elements in $\langle F,G \rangle$ depending upon the Nielsen-Thurston type of $F,G$.
\begin{lem}
For $g\geq 2$, consider a (non-cyclic) metacyclic subgroup $\langle F,G\rangle < \map(S_g)$ that admits the presentation $\langle F,G ~|~ G^{-1}FG=F^k \rangle$, where $k=\pm 1$. Then every nontrivial element of $\langle F,G \rangle$ is a reducible mapping class of infinite order.
\end{lem}
\begin{proof}
From Lemma \ref{lem:relation}, $F^iG^j=G^jF^{ik^j}$, so every element of $\langle F,G\rangle$ is of the form $G^iF^j$ for some integers $i,j$. For $G^iF^j\in \langle F,G\rangle$, from Lemma \ref{lem:relation}, we have
\begin{equation*}
(G^iF^j)^{\ell}=G^{i\ell}F^{j(1+k^i+k^{2i}+\cdots + k^{i(\ell-2)}+k^{i(\ell-1)})}.
\end{equation*}
It follows that every nontrivial element $G^iF^j\in \langle F,G \rangle$ has infinite order. Furthermore, by Theorem \ref{thm:main_thm2}, neither $G$ nor $F$ can be pseudo-Anosov mapping classes. When $G$ and $F$ are infinite order reducibles, from Theorem \ref{thm:main_thm1}, it follows that every nontrivial element of $\langle F,G \rangle$ preserves the multicurve $\mathcal{C}(F)\cup \mathcal{C}(G)$, and hence, is a reducible mapping class.
\end{proof}

\begin{lem}
For $g\geq 2$, consider a metacyclic subgroup $\langle F,G\rangle < \map(S_g)$ that admits the presentation $\langle F,G ~|~ F^n=1, G^{-1}FG=F^k \rangle$, where $n\geq 3$ and $k\in \Z_n^{\times}\setminus \{1\}$. Every nontrivial element of $\langle F,G \rangle$, except the powers of $F$, is of the same Nielsen-Thurston type as $G$.
\end{lem}
\begin{proof}
We note that $G$ can be either pseudo-Anosov or reducible of infinite order. When $G$ is reducible, from Theorem \ref{thm:main_thm1}, we have $F(\mathcal{C}(G))=\mathcal{C}(G)$. Moreover, for $i\neq 0$, we consider $G^iF^j\in \langle F,G \rangle$ and set $\ell=|F||k|$. Then from Lemma \ref{lem:relation}, we have
\begin{align*}
(G^iF^j)^{\ell}&=G^{i\ell}F^{j(1+k^i+k^{2i}+\cdots + k^{i(\ell-2)}+k^{i(\ell-1)})}\\
&= G^{i\ell}F^{j|F|(1+k^i+k^{2i}+\cdots + k^{i(\ell-2)}+k^{i(|k|-1)})}=G^{i\ell}.
\end{align*}
Hence, every nontrivial element of $\langle F, G\rangle$, except the powers of $F$, have same Nielsen-Thurston type as $G$.
\end{proof}

\begin{lem}
\label{lem:pa_periodic}
For $g,m\geq 2$, consider a metacyclic subgroup $\langle F,G\rangle < \map(S_g)$ that admits the presentation $\langle F,G ~|~ G^m=1, G^{-1}FG=F^k \rangle$, $k=\pm 1$ (for $k=-1$, $m$ is even).
\begin{enumerate}[(i)]
\item If $\langle F,G\rangle$ is abelian, then every nontrivial element of $\langle F,G \rangle$, except the powers of $G$, has the same Nielsen-Thurston type as $F$.
\item If $\langle F,G \rangle$ is non-abelian, then for integers $i,j$, where $j\neq 0$, $G^iF^j$ has the same Nielsen-Thurston type as $F$ if $i$ is even, and $G^iF^j$ is periodic of order $|G^i|$ when $i$ is odd. Furthermore, for $i$ odd and $j$ even, $G^iF^j$ is conjugate to $G^i$.
\end{enumerate}
In particular, when $F$ is pseudo-Anosov, $F^2$ can be written as a product of two periodic mapping classes of the same order.
\end{lem}
\begin{proof}
When $\langle F,G \rangle$ is abelian, it follows that every nontrivial element of $\langle F,G \rangle$, except the powers of $G$, is of infinite order of the same Nielsen-Thurston type as of $F$. We now consider the case when $\langle F, G \rangle$ is non-abelian. Since $G^2$ commutes with $F$, $G^{2i}F^j$ has same Nielsen-Thurston type as $F$, where $j\neq 0$. By Lemma \ref{lem:relation}, $(G^iF^j)^2=G^{2i}F^{(1+({-1})^i)}=G^{2i}$ if and only if $i$ is odd. Thus, it follows that, for $j\neq 0$, $G^iF^j$ is periodic of order $|G^i|$ if and only if $i$ is odd. When $j$ is even and $i$ is odd, we have $F^{j/2}(G^iF^j)F^{-j/2}=F^{j/2}G^iF^{j/2}=G^iF^{(j[1+(-1)^i])/2}=G^i$. Therefore, $G^iF^{j}$ is conjugate to $G^i$.
\end{proof}

\subsection{Centralizers of irreducible periodic mapping classes}
\label{subsec:centralizer}
In this subsection, we describe the centralizers of irreducible periodic mapping classes in $\map(S_g)$. 
\begin{prop}
\label{prop:centralizer}
For $g\geq 2$, let $F\in \map(S_g)$ be an irreducible periodic mapping class with $D_F=(n,0;(c_1,n_1),(c_2,n_2),(c_3,n_3))$. Let $H$ be the centralizer of $F$ in $\map(S_g)$.
\begin{enumerate}[(i)]
\item If either $n>2g+2$, or the $(c_i,n_i)$ are all distinct for $i=1,2,3$, then $H=\langle F \rangle$. 
\item If $n \leq 2g+2$ and $(c_i,n_i)=(c_j,n_j)$ for some $i,j\in \{1,2,3\}$ and $i\neq j$, then $H=\langle F \rangle \times \langle i \rangle
$, where $i$ is a hyperelliptic involution.
\end{enumerate}
\end{prop}
\begin{proof}
Since $F$ is irreducible, by Theorem \ref{thm:gilman}, $\Orb_F\approx S_{0,3}$. By Definition \ref{defn:data_set} (\ref{eqn:r-h}), we have $n\geq 2g+1$. For $G\in H\setminus \langle F\rangle$, by Lemma \ref{lem:ind_auto}, there exist $\bar{\G}\in \aut(\Orb_{F})$ induced by $\G$. Since $\map(S_{0,3})\cong \Sigma_3$ (where $\Sigma_3$ is the permutation group on three letters), it follows that $|G|=2$ or $3$. 

First, we consider the case when $|G|=3$. Since $\bar{\G}\in \aut(\Orb_{F})$, $\bar{\G}$ permutes the three cone points of $\Orb_{F}$, which implies that $(c_i,n_i)=(c_j,n_j)$ for all $i,j\in \{1,2,3\}$. By Definition \ref{defn:data_set} (\ref{eqn:lcm}), we have $(c_i,n_i)=(c_1,n)$ for every $i$. Moreover, by Definition \ref{defn:data_set} (\ref{eqn:angle_sum}), we have $3c_1\equiv 0 \pmod n$, which is impossible (as $n \geq 4$).

We now consider the case when $|G|=2$. Since $\bar{\G}\in \aut(\Orb_{F})$, $\bar{\G}$ permutes two cone points of $\Orb_{F}$ and fixes the third one. By Definition \ref{defn:data_set} (\ref{eqn:lcm}), we have $(c_i,n_i)=(c_j,n_j)=(c_i,n)$, for some $i,j\in \{1,2,3\}$ and $i\neq j$. We note that such a $\bar{G}\in \Sigma_3$ is uniquely determined. Since $G\in H$, we have $\langle F,G\rangle \cong \mathbb{Z}_n \times \mathbb{Z}_2$. By a result of Maclachlan \cite{ab_bound}, it follows that $|\langle F,G\rangle|\leq 4g+4$, and this implies that $n\leq 2g+2$. Thus, ($i$) follows.

When $n=2g+1$, by Definition \ref{defn:data_set} (\ref{eqn:r-h}), we have $n_i=2g+1$ for every $i$. Since $\Gamma(S_g/\langle \F,\G\rangle)=(0;2,2g+1,4g+2)$, it follows that $\langle \F,\G \rangle$ is cyclic. From the theory developed in \cite{dhanwani}, it follows that $\G$ is a hyperelliptic involution. When $n=2g+2$, by Definition \ref{defn:data_set} (\ref{eqn:r-h}), we have $\Gamma(\Orb_F)=(0;g+1,2g+2,2g+2)$. Since $\bar{\G}\in \aut(\Orb_F)$, it follows that $\Gamma(S_g/\langle \F,\G\rangle)=(0;2,2g+2,2g+2)$. Again, from the theory developed in \cite{dhanwani}, the possible data sets for $G$ are:
\begin{enumerate}[(a)]
\item $D_{G_1}=(2,0;\underbrace{(1,2),(1,2),\dots,(1,2)}_{(2g+2) \text{ times}})$,
\item $D_{G_2}=(2,g/2;(1,2),(1,2))$ when $g$ is even, and
\item $D_{G_3}=(2,(g+1)/2),1;-)$ when $g$ is odd.
\end{enumerate}
Furthermore, it can be shown that $F^{g+1}G_1$ is conjugate to $G_2$ (resp. $G_3$) when $g$ is even (resp. odd). Further, we note that lifts of $\bar{G}$ are $GF^j$, where $1\leq j\leq n$. This concludes our argument for ($ii$). 
\end{proof}

\section*{Acknowledgment}
The authors are grateful to the referees for helpful suggestions that have significantly improved the exposition in this manuscript. The first author was supported by the Prime Minister Research Fellowship (PMRF) scheme instituted by the Ministry of Education, India.

\bibliographystyle{abbrv}
\bibliography{Metacyclic}
\end{document}